\documentclass[11pt,reqno]{amsart}
\usepackage{xifthen,setspace,bm}
\usepackage[utf8]{inputenc} 
\usepackage[T1]{fontenc} 
\usepackage{pkgfile}
\usepackage{cleveref}
\newcommand{\e}{\varepsilon}

\allowdisplaybreaks

\title[Large Deviations of the Spectral Edge in Sparse Random Graphs]{Spectral Edge in Sparse Random Graphs: Upper and Lower Tail Large Deviations}

\author[Bhattacharya]{Bhaswar B. Bhattacharya}
\address{B.\ B.\ Bhattacharya\hfill\break
	Department of Statistics\\ University of Pennsylvania\\ Philadelphia, PA 19104, USA.}
\email{bhaswar@wharton.upenn.edu}

\author[Bhattacharya]{Sohom Bhattacharya}
\address{S. \ Bhattacharya\hfill\break
	Department of Statistics\\ Stanford University\\ California, CA 94305, USA.}
\email{sohomb@stanford.edu}

\author[Ganguly]{Shirshendu Ganguly}
\address{S. Ganguly \hfill\break
	Department of Statistics\\ UC Berkeley \\ 
	Berkeley, California, CA 94720, USA.}
\email{sganguly@berkeley.edu}

\begin{document}

\maketitle

\begin{abstract} In this paper we consider the problem of estimating the joint upper and lower tail large deviations of the edge eigenvalues of an Erd\H{o}s-R\'enyi random graph $\cG_{n,p}$,
in the regime of $p$ where the edge of the spectrum is no longer governed by global observables, such as the number of edges, but rather by localized statistics, such as high degree vertices. Going beyond the recent developments in mean-field approximations of related problems, this paper provides a comprehensive treatment of the large deviations of the spectral edge in this entire regime, which notably includes the well studied case of constant average degree. In particular, for $r \geq 1$ fixed, we pin down the asymptotic probability that the top $r$ eigenvalues are jointly greater/less than their typical values by multiplicative factors bigger/smaller than $1$, in the regime mentioned above. The proof for the upper tail relies on a novel structure theorem, obtained by building on estimates in  \cite{KS}, followed by an iterative cycle removal process, which shows, conditional on the upper tail large deviation event, with high probability the graph admits a decomposition in to a disjoint union of stars and a spectrally negligible part. On the other hand, the key ingredient in the proof of the lower tail is a Ramsey-type result which shows that if the $K$-th largest degree of a graph is not atypically small (for some large $K$ depending on $r$), then either the top eigenvalue or the $r$-th largest eigenvalue is larger than that allowed by the lower tail event on the top $r$ eigenvalues, thus forcing a contradiction. The above arguments reduce the problems to developing a large deviation theory for the extremal degrees which could be of independent interest.
\end{abstract}


\section{Introduction}

Understanding asymptotic properties of spectral statistics arising from random matrices has been the subject of intense study in recent years. Within this theme, a particularly important research direction is in deriving large deviation principles (rare event  probabilities) for spectral functionals, such as the empirical spectral measure and the extreme eigenvalues. For Gaussian models, 
the large deviation of the spectral measure was derived in the seminal work \cite{freeentropy} and the analogous problem for the spectral norm was studied in \cite{aging}. For results in the related setting of Gaussian covariance or Wishart matrices see \cite{satya1,satya2}. However, the non-Gaussian cases presented significant new challenges. Among the relatively fewer results, major breakthroughs include Bordenave and Caputo's study LDP for empirical spectral measure for, Wigner matrices with stretch exponential tails in \cite{bordenave1}. The corresponding results for the spectral norm were obtained in \cite{augeri2}. Very recently, Guionnet and Husson \cite{guionnet1} established an LDP for the
largest eigenvalue of Rademacher matrices and more generally sub-Gaussian Wigner matrices. 

Another natural class of random matrix models are those arising from random graphs, which are also widely studied and have witnessed a slew of applications in a range of different fields. Recent developments in spectral graph theory and theoretical computer science, relating geometric properties of graphs to their spectrum, have made the study of the spectral properties of these random matrices interesting and important. Particularly noteworthy developments include, among others, establishing refined bulk and edge universality properties dictated by Gaussian ensembles, for the spectrum of random graphs of average degree at least logarithmic in the graph size  (cf. \cite{yau, erdHos2013spectral} and the references therein). For sparser graphs, which includes the case of constant average degree, however, there have been much fewer results, notwithstanding the breakthroughs in \cite{edge_knowles, bbk_dense, bbk, KS} which looked at the edge of the spectrum. Also relevant are the beautiful results of \cite{bordenave2017mean} and \cite{empirical_neighborhood}, which studied continuity properties of the limiting spectral measure and developed a large deviation theory of the related local limits, respectively.  

In this paper we will study the upper and lower tail large deviations of the extreme eigenvalues of the adjacency matrix of the sparse Erd\H os-R\'enyi  random graph in the regime where they  are governed by localized statistics, such as high degree vertices.  To this end, let $\cG_{n,p}$ denote the Erd\H os-R\'enyi random graph on $n$ vertices with edge probability $p \in (0, 1)$, and $\lambda_1(\cG_{n,p}) \geq \lambda_2(\cG_{n,p}) \geq \cdots \geq \lambda_n(\cG_{n,p})$ be the eigenvalues, arranged in non-increasing order. The typical behaviors of  the edge eigenvalues of $\cG_{n,p}$ are, in general, well-understood \cite{bbk,eigenvalues_random_matrices,KS,vu_spectral_norm}. In particular, the breakthrough result of Krivelevich and Sudakov \cite[Theorem 1.1]{KS} shows that, for any $p \in (0, 1)$, with high probability
\begin{align}\label{eq:lambda1_expectation}
\lambda_1(\cG_{n, p})= (1+o(1)) \max \left\{\sqrt{d_{(1)} (\cG_{n, p})}, np\right\},\
\end{align}
where $d_{(1)}(\cG_{n, p})$ is the maximum degree of the graph $\cG_{n, p}$. This dichotomy has several consequences as well in the behavior of eigenvectors, as elaborated later. This result has been substantially generalized to other edge eigenvalues and to inhomogeneous random graph models by Benaych-Georges, Bordenave,  and Knowles \cite{bbk_dense,bbk}. General concentration inequalities for the extreme eigenvalues of $\cG_{n, p}$ are also well-known \cite{AKV} (see \cite{ev_concentration} for a recent improvement for the largest eigenvalue). 

More recently, there has been several advances in the large deviations of spectral observables of $\cG_{n,p}$. This began with Chatterjee and Varadhan \cite{CV12} where, building on their seminal work \cite{CV11}, the authors proved a large deviation principle for the entire spectrum of $\cG_{n, p}$ at scale  $np$, in the dense case, where $p$ is fixed (does not depend on $n$). However, these techniques turned out to be inadequate to handle the sparse regime, $p=p(n) \rightarrow 0$. The first major breakthrough into the sparse regime was made by Chatterjee and Dembo \cite{CD16}, which sparked a series of developments in  understanding the large deviations principle for various functionals of interest for sparse random graphs, as well as related natural statistical mechanics models \cite{augeri1,austin,BM17,eldan,yan}. Here, a fundamental example of interest is the large deviations of counts of small subgraphs, such as cycles, in $\cG_{n, p}$. In a series of exciting developments \cite{augeri1,ab_rb,BGLZ,CV11,CD18,upper_tail_localization,LZ-dense,LZ-sparse}, the precise behavior of the upper tail rate function for cycle counts in $\cG_{n,p}$ have been pinned down. This has a natural implication in the large deviations of spectral statistics via the relation between spectral moments and cycle counts. This strategy was adopted in a previous work by the first and the third author in \cite{uppertail_eigenvalue} to obtain the precise asymptotics for the upper tail large deviations of the spectral norm. However such arguments only extended to $p$ going to zero at a rate slower than $1/\sqrt n$,  since for sparser graphs, atypical behavior of the spectral norm fails to have a discernible effect on the cycle statistics.

In fact, as is evident from \eqref{eq:lambda1_expectation}, there is a threshold of sparsity (made precise later in \eqref{eq:lambda1_N}), above which the spectral norm is governed by the total number of edges and a fully delocalized eigenvector, while below the threshold it is determined by the highest degree vertex and a fully localized eigenvector. The regime of focus in the present paper is the latter, which, in particular, includes the case of constant average degree and is of interest in several areas of statistical physics.  Addressing the inability of the current techniques in treating such sparse random graphs, we develop new methods to understand the large deviation behavior of the edge eigenvalues in both the upper and lower tail regimes which, unsurprisingly, have quite different asymptotics compared to those obtained in \cite{uppertail_eigenvalue,CV12,CD18}. The main results of the paper are the joint upper and lower tail large deviations for the extreme eigenvalues 
of $\cG_{n, p}$ in the entire regime where the largest eigenvalues are typically determined by the largest degrees in the graph. 
In the process, we develop a large deviations theory for the extreme degrees of $\cG_{n, p}$, a relatively simple, yet useful, result which might be of independent interest. Our proofs also unearth interesting structural results on the random graphs and the relevant eigenvectors under such large deviation events. The formal statement of the theorems and key ideas of the proofs are given below in Section \ref{sec:eigenvalues} and Section \ref{sec:pf_outline}, respectively.

\subsection{Statements of the Main Results}\label{sec:eigenvalues}  Let $\sG_n$ denote the set of all simple, undirected graphs on $n$ vertices labelled $[n]:=\{1, 2, \ldots, n\}$. For $G \in \sG_n$ denote by $A(G)=((a_{ij}))_{1 \leq i, j \leq n}$ the adjacency matrix of $G$, that is $a_{ij}=1$, if $(i, j)$ is an edge in $G$, and 0 otherwise.
For $F \in \sG_n$, since $A(F)$ is a self-adjoint matrix, denote by $\lambda_1(F) \geq \lambda_2(F) \geq \cdots \geq \lambda_n(F)$ its eigenvalues in non-increasing order, and let $||F||:=||A(F)||=\max\{|\lambda_1(F)|, |\lambda_n(F)|\}$ be the operator norm of $A$. Various properties of the eigenvalues of a graph are often closely related to its degrees. To this end, for a graph $F \in \sG_n$, denote by $d_{v}(F)$ the degree of the vertex $v \in [n]$. Moreover, let $d_{(1)}(F) \geq d_{(2)}(F) \geq \cdots \geq d_{(n)}(F)$ be the degrees of the graph $F$ in non-increasing order.  
 
In this paper we are interested in the large deviations of the edge eigenvalue of the sprase Erd\H os-R\'enyi random graph $\cG_{n, p}$, where $p=p(n) \in (0, 1)$ can depend on $n$. As mentioned above, Krivelevich and Sudakov \cite[Theorem 1.1]{KS} derived the typical value of the largest eigenvalue $\lambda_1(\cG_{n, p})$ in the entire range of $p$, (recall \eqref{eq:lambda1_expectation}). Using the asymptotics of the typical value of the maximum degree $d_{(1)} (\cG_{n, p})$ (this is discussed later in Section \ref{sec:degree_preliminaries}, also cf. \cite[Lemma 2.2]{KS} and \cite[Proposition 1.9]{bbk}), it follows from \eqref{eq:lambda1_expectation} above that $\lambda_{(1)}(\cG_{n, p})$ has a qualitative change in behavior in the threshold regime where $n p$ is comparable to $\sqrt{\log n/\log \log n }$. More precisely (cf. \cite[Proposition 1.9]{bbk}, and \cite[Theorem 1.1 and Lemma 2.1]{KS}),\footnote{For two nonnegative sequences $\{a_n\}_{n\geq 1}$ and $\{b_n\}_{n\geq 1}$, $a_n = o(b_n)$ means $\lim_{n \rightarrow \infty} a_n/b_n=0$, and $a_n = \omega(b_n)$ means $\lim_{n \rightarrow \infty} a_n/b_n= \infty$. Similarly,  $a_n \ll b_n$  means $a_n = o(b_n)$, and $a_n \gg b_n$ means $a_n=\omega(b_n)$. Moreover, for two sequence of positive random variables $\{X_n\}_{n\geq 1}$ and $\{Y_n\}_{n\geq 1}$ defined on the same probability space, we write $X_n= (1+o(1)) Y_n$ if $X_n/Y_n$ converges to 1 in probability.}   
\begin{equation}\label{eq:lambda1_N}
\lambda_{(1)}(\cG_{n, p}) =\left\{
\begin{array}{cc}
(1+o(1)) \sqrt{d_{(1)}(\cG_{n, p})}  &   \text{ if }  0 < p \ll \frac{1}{n}\sqrt{\frac{\log n}{\log \log n }}, \\ 
(1+o(1)) np  &   \text{ if }  \frac{1}{n}\sqrt{\frac{\log n}{\log \log n}} \ll p < 1. 
\end{array}
\right.
\end{equation}
This shows that the typical value of the largest eigenvalue is governed by the maximum degree of the random graph below the threshold, whereas above the threshold it is determined by the total number of edges in the graph. In this paper, we will study the upper and lower tail large deviations of the $r$ largest eigenvalues of $\cG_{n, p}$ below the threshold. More specifically, we will consider the following sparsity regime: 
\begin{equation}\label{eq:rangep}
\log n \gg \log (1/np) \quad \text{ and } \quad n p \ll \sqrt{\frac{\log n}{\log \log n}}.
\end{equation} 
It can be shown that throughout this regime, for $r \geq 1$ fixed and $1 \leq a \leq r$,
\begin{align}\label{eq:Lp} 
\lambda_{a}(\cG_{n, p})=(1+o(1)) \sqrt{L_p}, \text{ where } L_p=\frac{\log n}{\log \log n -\log (np)}.
\end{align}
This result is proved in \cite[Proposition 1.9, Corollary 1.13]{bbk} with the first condition in \eqref{eq:rangep} replaced by $np \gtrsim 1$. The slightly weaker constraint $\log n \gg \log (1/np)$, which can be rewritten as $p \gg \frac{1}{n^{1+o(1)}}$,  allows a wider range of sparsity. (The validity of \eqref{eq:Lp} in this wider range is verified below in Lemma \ref{lem:degreep}.) Moreover, the edge eigenvalues remain bounded with high probability whenever 
$\log n \lesssim \log (1/np)$, whereas they diverge for $\log n \gg \log (1/np)$ (see Lemma \ref{lem:degreep}). Hence, the regime in \eqref{eq:rangep} covers the complete range of $p$ where the typical values of the edge eigenvalues are unbounded and governed by the largest degrees. 
We formally state our results for the joint upper and lower tail large deviations  for the $r$ largest eigenvalues in Section \ref{sec:eigenvalue_largest}. The corresponding results for the $r$ largest degrees are given in Section \ref{sec:degree_largest}. We discuss the analogous questions for the largest eigenvalue in the denser regime $np \gg \sqrt{\log n/\log \log n}$ in Section \ref{sec:lambda1}. This discussion includes an observation that a recent result of Cook and Dembo \cite{CD18}, regarding the lower tail large deviation of $\lambda_1(\cG_{n,p})$, in fact, holds in greater generality (see Proposition  \ref{ppn:lt_lambda1}). Several future research directions are discussed as well.

\subsubsection{\emph{\textbf{Joint Large Deviation of the Edge Eigenvalues}}}
\label{sec:eigenvalue_largest}

For $r \geq 1$, define the {\it upper tail event for the top $r$ eigenvalues} as
\begin{align}\label{eq:ut_delta}
\mathrm{UT}_r(\bm \delta) =\{F \in \sG_n: \lambda_1(F) \geq (1+\delta_1)  \sqrt{L_p}, \ldots, \lambda_r(F) \geq (1+\delta_r)  \sqrt{L_p}\},
\end{align}
where $\bm \delta = (\delta_1, \delta_2, \cdots, \delta_r)$ for some $\delta_1 \geq \delta_2 \geq \ldots \geq \delta_r >0$. The following theorem gives the precise asymptotics of the log-probability of the upper tail event in the regime of interest.  The proof is given in Section \ref{sec:ut_pf}. 

\begin{thm}\label{thm:1ut}
For $p$ as in \eqref{eq:rangep}, $r \geq 1$ fixed, and $\delta_1 \geq \delta_2 \geq \ldots \geq \delta_r>0$, 
	\begin{align}\label{eq:1ut}
			\lim_{n \rightarrow \infty}\frac{-\log \P( \cG_{n, p} \in \mathrm{UT}_r(\bm \delta)  ) }{\log n}  = \sum_{a=1}^{r}(2\delta_a+\delta^2_a), 
	\end{align}
where $\mathrm{UT}_r(\bm \delta)$ is as defined above in \eqref{eq:ut_delta}. 
\end{thm}
Note that the above shows that the upper tail probabilities decay as polynomials in $n$, unlike for $1/\sqrt{n} \ll p \ll 1$, where the results of \cite{uppertail_eigenvalue} imply that for $r=1,$ the upper tail  probabilities are of the form $e^{-(1+o(1)) c(\delta_1) n^2p^2\log(1/p))}$, for some constant $c(\delta_1) > 0$, depending on $\delta_1$. Discussions about the intermediate regime $\sqrt{\log n /\log \log n }/n \ll p \ll 1/\sqrt n$ appear in Section \ref{sec:lambda1}.

Next, we consider the lower tail probability.  To this end, for $0 < \delta_1 \leq \delta_2 \leq \cdots \leq \delta_r < 1$, denote the {\it lower tail event for the top $r$ eigenvalues} as, 
\begin{align}\label{eq:lt_delta}
\mathrm{LT}_r(\bm \delta) =\{F \in \sG_n: \lambda_1(F) \leq (1-\delta_1)  \sqrt{L_p},  \ldots, \lambda_r(F) \leq (1-\delta_r)  \sqrt{L_p}\}.
\end{align}
The lower tail asymptotics, which happens in log-log-scale, is given in the theorem below. The proof is given in Section \ref{sec:lt_pf}. 

\begin{thm}\label{thm:lt}
For $p$ as in \eqref{eq:rangep}, $r \geq 1$ fixed, and $0 < \delta_1 \leq \delta_2 \leq \cdots \leq \delta_r < 1$, 
	\begin{equation}\label{eq:1lt}
		\lim_{n \rightarrow \infty}\frac{1}{\log n} \left( \log \log \frac{1}{\P( \cG_{n, p} \in \mathrm{LT}_r(\bm \delta)  )} \right) =  2\delta_r - \delta_r^2,
	\end{equation}
where $\mathrm{LT}_r(\bm \delta)$ is as defined above in \eqref{eq:lt_delta}. 	
\end{thm}

Note that the results above show, while the upper tail probabilities decay as polynomials in $n$, the lower tail probabilities decay as stretch exponentially $n$.  This is consistent with many other natural settings where upper tail events lead to localized effects, whereas lower tail events force a more global change and, hence, are typically much more unlikely. Perhaps, the most well known example of this is in the tail behavior of the Tracy-Widom distribution function $F_{TW}(\cdot)$, which arises as the scaling limit of the largest eigenvalue of a Gaussian Unitary Ensemble (GUE) matrix. Here, the logarithm of the right tail of the Tracy-Widom distribution function $-\log(1-F_{TW}(x))$ decays as $x^{3/2}$, while the logarithm of the left tail $-\log F_{TW}(x)$ decays as  $x^3$.

\subsubsection{\emph{\textbf{Joint Large Deviation of the Largest Degrees}}}
\label{sec:degree_largest}

As alluded to before, a simple but useful ingredient in the above proofs is the large deviation of the $r$ largest degrees of $\cG_{n, p}$. This is what is developed next. To this end, given $\varepsilon_1 \geq \varepsilon_2 \geq \cdots \geq \varepsilon_r >0$, denote the {\it upper tail event for the $r$ largest degrees} by 
\begin{align}\label{eq:degree_ut_delta}
\mathrm{degUT}_r(\bm \varepsilon):=\{F \in \sG_n: d_{(1)}(F) \geq (1+\varepsilon_1)  L_p,  \ldots, d_{(r)}(F) \geq (1+\varepsilon_r)  L_p\}.
\end{align}
Similarly, for $0 < \varepsilon_1 \leq \varepsilon_2 \leq \cdots \leq \varepsilon_r < 1$, denote the {\it lower tail event for the $r$ largest degrees} by 
\begin{align}\label{eq:degree_lt_delta}
\mathrm{degLT}_r(\bm \varepsilon):=\{F \in \sG_n: d_{(1)}(F) \leq (1-\varepsilon_1)  L_p,  \ldots, d_{(r)}(F) \leq (1-\varepsilon_r)  L_p\}.
\end{align}
 
The following two propositions state the asymptotics of the probabilities of the upper and lower tail events respectively.

\begin{ppn}\label{ppn:degree_ut}
For $\log n \gg \log (1/np)$ and $n p \ll \log n$, $r \geq 1$ fixed, and $\varepsilon_1\geq \ldots \geq \varepsilon_r >0$, 
	\begin{equation}\label{eq:degree_ut}
	\lim_{n \rightarrow \infty} \frac{-\log \P(\cG_{n, p} \in \mathrm{degUT}_r(\bm \varepsilon) )}{\log n}= \sum_{a=1}^{r}\varepsilon_a. 
	\end{equation}
\end{ppn}

The proof is given in Section \ref{sec:dutpf}. The upper bound on the tail probability, is obtained by a union bound and an application of Chernoff's inequality for tails of the binomial distribution. For the matching lower bound, we fix a small $\kappa \in (0, 1)$, and split the vertex set of $\cG_{n, p}$ into two parts of sizes $n- \ceil{\kappa n}$ and $\ceil{\kappa n}$, and consider the degrees of the vertices from the larger part in to the smaller part. This has the advantage of making the degrees independent, hence, the resulting probabilities can be explicitly computed.

\begin{ppn}\label{ppn:deglt} 
For $\log n \gg \log (1/np)$ and $n p \ll \log n$, $r \geq 1$ fixed, and $0< \varepsilon_1\leq \ldots \leq \varepsilon_r <1$, 
	\begin{equation}\label{eq:deg_lt}
	\lim_{n \rightarrow \infty} \frac{1}{\log n} \left( \log\log \frac{1}{\P(\cG_{n, p} \in \mathrm{degLT}_r(\bm \varepsilon) )} \right)= \varepsilon_r,
	\end{equation}
where $\mathrm{degLT}_r(\bm \varepsilon)$ is as defined in \eqref{eq:degree_lt_delta}.	
\end{ppn}

The proof of this proposition is given in Section \ref{sec:degltpf}. As before, the upper  bound on the probability is obtained by a union bound and tail estimates of the binomial distribution. For the lower bound, we observe that the lower tail event for the degrees is a decreasing event (see Lemma \ref{lem:degfkg}), which allows us to invoke the  Fortuin-Kasteleyn-Ginibre (FKG) inequality, which implies that the smallness of individual degrees are positively correlated.

\begin{remark}\label{remark:degree_range} 
{\em Note that in both the above results we consider $n p \ll \log n$, instead of $n p \ll \sqrt{\log n/\log \log n}$ as in \eqref{eq:rangep}. This is because, even  though the typical value of the largest eigenvalue undergoes a transition from $\sqrt{L_p}$ to $np$, when $n p$ is comparable to $\sqrt{\log n/\log \log n}$, the largest degrees continue to be determined by $L_p$ as long as $n p \ll \log n$ (see Lemma \ref{lem:degreep} below), and the results above give their joint large deviations probabilities in this entire regime. } 
\end{remark}

\subsection{Proof Ideas}
\label{sec:pf_outline}

Before proceeding to the technical details, we sketch here the key ideas involved in the proofs of Theorem \ref{thm:1ut} and Theorem \ref{thm:lt}.  \\ 

\noindent \textsf{The Upper Tail}: The proof of the upper tail is split in two cases: (a) where  $e^{-(\log \log n)^2} \leq n p \ll \sqrt{\log n/\log \log n}$ and (b) for  $(\log \log n)^2  \leq \log(1/n p)  \ll \log n$. Building on the proof of Sudakov and Krivelevich \cite{KS}, the key technical work here goes into obtaining a structure theorem for the graph, conditioned on the large deviation event. Below we outline the key steps in the former case, and then explain the additional arguments needed to treat the latter case. 

\begin{itemize}

\item The first step is to observe that it is unlikely for the random graph to have `large' cycles because this probability is super-polynomially small in $n$ (Lemma \ref{lem:largecycle}), making this atypical even under the upper tail event which has probability polynomial in $n$. We then prove that no vertex is incident on too many edge-disjoint `small' cycles (Lemma \ref{lem:blem}). Hence, by iteratively removing cycles from the graph one can decompose the graph into a forest and a subgraph with a small spectral norm.  Moreover, conditioned on the upper tail event, the graph is unlikely to have many paths of length two between high degree vertices, and, hence, can be pruned again to obtain a disjoint union of high-degree star graphs and a graph with negligible spectral contribution (Lemma \ref{lem:decompose}). Finally, since the spectrum of a star graph is explicit, this allows us to conclude that the upper tail event for the top $r$ eigenvalues induces $r$ vertices with higher than typical degrees.  

\item In the second case, the argument takes a slightly different route, because few estimates used in the previous case cease to hold. However, in this case, one can directly remove all cycles and decompose the graph in to a forest and a subgraph with small spectral norm. Then we use estimates on the sizes of the tree components (Lemma \ref{lem:dlem}), as well as properties of the spectrum of trees to argue that the the upper tail event for the $r$ extreme eigenvalues induces $r$ disjoint trees of specific sizes, and the leading eigenvalues all come from the different tree components. 
\end{itemize}
In both the cases above, matching lower bounds are obtained via appropriate constructions which relies on the structural inputs from the arguments in the upper bound.  We also discuss,  in Section \ref{sec:pf_concentration}, another potential approach to proving the upper bound on the upper tail large probability using the strong concentration results of spectral norm of non-homogeneous sparse graphs proved in \cite{vershynin}. While several details remain to be verified, to illustrate the promise of this method we provide a proof sketch for the case $r=1$ and $p=c/n$. \\ 

\noindent \textsf{The Lower Tail}: The upper bound on the lower tail probability relies on an interesting Ramsey-type result for the extreme eigenvalues by relating them to the degrees: Namely, given $\varepsilon \in (0, 1)$ small enough (depending polynomially on $1/r$), we show that if the $K$-th largest degree (for some $K$ which is exponentially large in a polynomial in $1/\varepsilon$, arising from the bounds on the classical Ramsey-type functions) of the graph is bigger than $(1-\delta_r + \varepsilon)^2 L_p$, then either the largest eigenvalue $\lambda_1(\cG_{n, p}) > \sqrt{2 {L_p}}$ or the $r$-th largest eigenvalue $\lambda_r (G) \geq (1-\delta_r + \varepsilon/2) \sqrt{L_p}$ (see Lemma \ref{lm:r_I} for the precise statement). This contradicts the assumption of the lower tail event on the first $r$ eigenvalues, and shows that the lower tail probability is bounded above by the probability that $d_{(K)} \leq (1-\delta_r + \varepsilon)^2 L_p$. Since the lower probabilities of events are stretch exponentially small, the fact that our argument controls $d_{(K)}(\cG_{n, p})$, and not $d_{(r)}(\cG_{n, p})$, does not worsen our probability estimates in the leading order. The arguments use a variety of spectral graph theoretic observations coupled with an application of the classical Ramsey's theorem, and could be potentially useful in other applications. The lower bound on the lower tail probability relies on a construction which invokes the FKG inequality similar to previous applications (Section \ref{sec:lt_I}).

\section{Large deviations of the largest degrees: Proofs of Propositions \ref{ppn:degree_ut} and \ref{ppn:deglt}}
\label{sec:degree_pf}

In this section we prove the joint large deviation of the $r$ largest degrees of $\cG_{n, p}$. The upper tail (Proposition \ref{ppn:degree_ut}) is proved in Section \ref{sec:dutpf} and the lower tail (Proposition \ref{ppn:deglt}) is proved in Section \ref{sec:degltpf}. We begin with some technical preliminaries in Section \ref{sec:degree_preliminaries}.

\subsection{Preliminaries}
\label{sec:degree_preliminaries}

We begin by recording some standard asymptotic notation. For two positive sequences $\{a_n\}_{n\geq 1}$ and $\{b_n\}_{n\geq 1}$, $a_n = O(b_n)$ means $a_n \leq C_1 b_n$, $a_n = \Omega(b_n)$ means $a_n \geq C_2 b_n$, and $a_n = \Theta(b_n)$ means $C_2 b_n \leq a_n \leq C_1 b_n$, for all $n$ large enough and positive constants $C_1, C_2$. Similarly, $a_n \lesssim b_n$ means $a_n=O(b_n)$, and $a_n \gtrsim b_n$ means  $a_n=\Omega(b_n)$, and subscripts in the above notation,  for example $O_{\square}$ or $\lesssim_\square$,  denote that the hidden constants may depend on the subscripted parameters. 

We begin with an useful lemma which relates the second condition in \eqref{eq:rangep} with the condition $np \ll \sqrt{L_p}$, where $L_p$ is defined in \eqref{eq:Lp}. The proof is given in Appendix \ref{sec:sqrtLp_pf}. 

\begin{lem}\label{equivalence} 
The condition $n p \ll \sqrt{\log n/\log \log n}$ implies the condition $np \ll \sqrt{L_p}$. Moreover, if $n p \ll  {\log n}$, then the condition $np \ll \sqrt{L_p}$ implies $n p \ll \sqrt{\log n/\log \log n}$. Finally, the condition $n p \ll  {\log n}$ implies the condition $np \ll {L_p}.$ 
\end{lem}

An important step towards understanding the large deviation probabilities of the largest degrees and edge eigenvalues is to understand their corresponding typical values in the regime of interest. In this regard, it is  well-known \cite[Lemma 2.2]{KS} that the maximum degree is $d_{(1)}(\cG_{n, p})=(1+o(1))\Delta_p$, where  
\begin{align}\label{eq:degreep}
\Delta_p := \max \left\{s: n \binom{n-1}{s}p^s(1-p)^{n-s} \geq 1 \right\}. 
\end{align} 
Note that $\Delta_p$ is the largest $s$ for which the expected number of vertices with degree $s$ is at least $1$.

The following lemma shows that 
\begin{equation}\label{relation345}
\Delta_p=(1+o(1))L_p
\end{equation} and for any fixed $r,$ $d_{(a)}(\cG_{n, p})=(1+o(1)) L_p$, for $1\leq a \leq r$, where $L_p$ is defined in \eqref{eq:Lp}, whenever $\log n \gg \log (1/np)$ and $n p \ll  \log n$. It also lists other properties of the largest degrees and eigenvalues, which we will often rely on in our proofs. 

\begin{lem}\label{lem:degreep} Let $\Delta_p$ be as defined in \eqref{eq:degreep}. Then the following hold: 
\begin{enumerate}
\item[(a)] For $\log n \lesssim \log (1/np)$, $\Delta_p=O(1)$.
\item[(b)] For $\log n \gg \log (1/np)$ and $n p \ll  \log n$, $\Delta_p=(1+o(1)) L_p$. 
\item[(c)] For $\log n \gg \log (1/np)$ and $n p \ll  \log n $, with high probability, $d_{(a)}(\cG_{n, p})=(1+o(1)) L_p$, for $1\leq a \leq r$. 
\item[(d)]  For $\log n \gg \log (1/np)$ and $n p \ll  \sqrt{\log n/\log \log n} $,  with high probability, $\lambda_a(\cG_{n, p})=(1+o(1)) \sqrt{L_p}$, for $1\leq a \leq r$. 
\end{enumerate}
\end{lem}

The proof of the lemma is given in Appendix \ref{sec:degreep_pf}. In fact, the results in Lemma \ref{lem:degreep}(b), (c), and (d) are proved in \cite[Corollary 1.13]{bbk} in the regime $1 \lesssim n p \ll  \log n $. The condition $\log n \gg \log (1/np)$ above allows for a slightly larger range of $p$. Note that Lemma \ref{lem:degreep}(a) and (b), combined show that $\log n \gg\log (1/np)$ is the threshold where the largest degrees transitions from being bounded to unbounded, since $L_p\gg 1$ precisely at this threshold. As mentioned before,  \eqref{eq:rangep} covers the full range where typically the largest eigenvalue is unbounded and governed by the maximum degree, and Lemma \ref{lem:degreep}(d) shows that throughout this regime the largest eigenvalue is asymptotically $\sqrt{L_p}$.

\subsection{Proof of Proposition \ref{ppn:degree_ut}}
\label{sec:dutpf}

We begin with the upper bound on $\P(\cG_{n, p} \in \mathrm{degUT}_r(\bm \varepsilon))$. Label the vertices of the graph $\cG_{n, p}$ by $\{1, 2, \ldots, n\}$. Note that 
$$\{\cG_{n, p} \in \mathrm{degUT}_r(\bm \varepsilon) \} \subseteq \left\{ \exists \{i_1, \ldots, i_r\} \subset [n]: d_{i_1}(\cG_{n, p}) \geq (1+\varepsilon_{i_1}) L_p, \ldots,  d_{i_r}(\cG_{n, p}) \geq (1+\varepsilon_{i_r}) L_p\right\}.$$
Then, by a union bound, 
\begin{equation*}
\P(\cG_{n, p} \in \mathrm{degUT}_r(\bm \varepsilon)) \leq n^r \P(d_{1}(\cG_{n, p}) \geq (1+\varepsilon_1) L_p, \ldots, d_{r}(\cG_{n, p}) \geq (1+\varepsilon_r)  L_p ).  
	\end{equation*}
For $1 \leq a \leq r$, denote by $d_{a}^+(\cG_{n, p})$ the number of edges from  the vertex labelled $a$ to the set of vertices $\{r+1,\ldots, n\}$. Note that $d_{a}(\cG_{n, p}) \geq (1+\varepsilon_a) L_p$ implies $d_{a}^+(\cG_{n, p}) \geq (1+\varepsilon_a) L_p-(r-1)$. Moreover, $d_{1}^+(\cG_{n, p}), d_{2}^+(\cG_{n, p}), \ldots, d_{r}^+(\cG_{n, p})$ are independent and $d_{a}^+(\cG_{n, p}) \sim \dBin(n-r, p)$ for $1\le a\le r$. Therefore,
\begin{align}\label{eq:degutpf_I}
\P(\cG_{n, p} \in  \mathrm{degUT}_r(\bm \varepsilon) ) & \leq  n^r \prod_{a=1}^{r} \P(d_{a}^+(\cG_{n, p})  \geq(1+\varepsilon_a) L_p-(r-1)) \nonumber \\ 
& \leq n^r \prod_{a=1}^{r}\P(\dBin(n, p)\geq(1+\varepsilon_a) L_p-r ).
\end{align}
Now, by Chernoff's bound, for $q > p$, $\P(\dBin(n, p) \geq n q) \leq e^{-n I_p(q)}$, where $I_p(x) := x \log\frac{x}{p}+ (1-x) \log\frac{1-x}{1-p}$ is the relative entropy function (cf. \cite[Lemma 4.7.2]{ash}). Using this in \eqref{eq:degutpf_I} above gives,  
\begin{align}\label{eq:degutpf_II}
\P(\cG_{n, p} \in  \mathrm{degUT}_r(\bm \varepsilon) ) \leq  n^r  \exp\left\{- n \sum_{a=1}^r I_{p}(p + g_p(\varepsilon_a)) \right\},  
\end{align} 		
where $g_p(\varepsilon_a)=\frac{(1+\varepsilon_a)L_p-r}{n}-p$. Using $L_{p}\gg 1$ (which follows from the assumption $\log n \gg \log(1/np)$) and $L_p \gg n p$ (which follows by Lemma \ref{equivalence}), we get $g_p(\varepsilon_a) \gg p$. Therefore, by estimates on the relative entropy function, see \cite[Lemma 3.3]{LZ-sparse}, 
\begin{align}\label{eq:h}
I_p(p+g_p(\varepsilon_a))= (1+o(1)) g_p(\varepsilon_a) \log(g_p(\varepsilon_a)/p). 
\end{align}
{Moreover, since $g_p(\varepsilon_a)=(1+o(1)) \frac{(1+\varepsilon_a)L_p}{n}$ and 
\begin{align}\label{eq:LplogLp}
 L_p \log((1+\varepsilon_a)L_p/np)&= \frac{\log(1+\varepsilon_a)+ \log\log n -\log(np) - \log(\log\log n -\log(np))}{\log\log n -\log(np)} \log n \nonumber \\
&= (1+o(1)) \log n, 
\end{align} 
where the last step uses $\log\log n -\log(np) \rightarrow \infty$,} since $np \ll \log n$ by  hypothesis. Then by \eqref{eq:degutpf_II} and  \eqref{eq:h},  
\begin{align*}
\P(\cG_{n, p} \in  \mathrm{degUT}_r(\bm \varepsilon) ) & \leq n^r \exp\left\{-\sum_{a=1}^{r}(1+\varepsilon_a)\log n+o(\log n) \right\} \nonumber \\ 
&=\exp\left\{-\left(\sum_{a=1}^r \varepsilon_a \right) \log n +o(\log n) \right\},
\end{align*}
which gives the upper bound on the log-probability as in \eqref{eq:degree_ut}.

We now proceed to lower bound $\P(\cG_{n, p} \in \mathrm{degUT}_r(\bm \varepsilon))$. To this end, fix $\kappa \in (0, 1)$ and  let 
\begin{align}\label{eq:V12}
V_1=\{1, 2, \ldots, n-\ceil{\kappa n} \} \quad  \text{and}  \quad V_2=[n]\backslash V_1,
\end{align} 
and consider the bipartite subgraph $F_\kappa$ of $\cG_{n, p}$ with vertex partition $V_1$ and $V_2$. Partition $V_1$ in to $r$ sets of approximately equal size, that is, 
$$V_1=V_{1, 1} \bigcup \cdots \bigcup V_{1, r},$$
where $V_{1, a} \bigcap V_{1, b} =\emptyset$, for $1 \leq a < b \leq r$, $|V_{1, a}|= \left\lceil\frac{n-\ceil{\kappa n}}{r} \right \rceil$, for $1 \leq a \leq r-1$. Define, for $1 \leq a \leq r$,  
$$d_{(V_{1, a})}(F_\kappa):=\max_{v \in V_{1, a}} d_{v}(F_\kappa),$$
the maximum degree in the graph $F_\kappa$ of a vertex in $V_{1, a}$. Then note that 
$$\bigcap_{a=1}^r \left\{d_{(V_{1, a})}(F_\kappa)\geq (1+ \varepsilon_a) L_p \right\} \subseteq 
\{\cG_{n, p} \in \mathrm{degUT}_r(\bm \varepsilon)\}.$$ 
Therefore, since the events $\{d_{(V_{1, 1})}(F_\kappa)\geq (1+ \varepsilon_1) L_p \}, \{d_{(V_{1, 2})}(F_\kappa)\geq (1+ \varepsilon_2) L_p \}, \ldots, \{d_{(V_{1, r})}(F_\kappa)\geq (1+ \varepsilon_r) L_p \}$ are independent, 
\begin{equation*}
\P( \cG_{n, p} \in \mathrm{degUT}_r(\bm \varepsilon) ) \geq \prod_{a=1}^r \P(d_{(V_{1, a})}(F_\kappa)\geq (1+ \varepsilon_a) L_p ).
\end{equation*}
The lower bound for $\P(\cG_{n, p} \in \mathrm{degUT}_r(\bm \varepsilon))$, as in \eqref{eq:degree_ut}, now follows from the lemma below, completing the proof of Proposition \ref{ppn:degree_ut}. \hfill $\Box$ 

\begin{lem} For any $\kappa \in (0, 1)$ and $1 \leq a \leq r$, 
$$\P(d_{(V_{1, a})}(F_\kappa)\geq (1+ \varepsilon_a) L_p ) = e^{-\varepsilon_a \log n + o(\log n)},$$
where $F_\kappa$ is as defined above.
\end{lem}

\begin{proof} Throughout the proof we denote $N:=\ceil{\kappa n}$. Recalling the definition of $V_2$ from \eqref{eq:V12}, note that $\{d_v(F_\kappa), v \in V_{1, a} \}$ are independent $\dBin(N, p)$ random variables. This implies, 
\begin{equation}\label{eq:dV_I}
\P(d_{(V_{1, a})}(F_\kappa)\geq (1+ \varepsilon_a) L_p ) = \P\left(\max_{v \in V_{1, a}} d_v(F_\kappa) \geq (1+\varepsilon_a)  L_p\right) = 1- (1- \theta_p)^{|V_{1, a}|},
\end{equation}
where $\theta_p=\P(d_1(F_\kappa) \geq (1+\varepsilon_a)  L_p)=\P(\dBin( N, p) \geq (1+\varepsilon_a)  L_p)$. {Now, using the inequality (cf. \cite[Lemma 4.7.2]{ash}), 
$$\P(\dBin(N,p) \geq K) \geq \frac{e^{-NI_p(K/N)} }{\sqrt{8K(1-K/N)}}, \quad \text{for } K > N p,$$ gives
\begin{align}
\theta_p  & \geq \frac{ e^{- N I_{p}\left( (1+ \varepsilon_a )L_p/N\right)} }{\sqrt{8(1+\varepsilon_a)L_p(1- (1+\varepsilon_a)L_p/N)}}  \tag*{ (since $L_p \gg np$ by Lemma \ref{equivalence})}\nonumber \\ 
&\geq  e^{-(1+\varepsilon_a)L_p \log\left((1+\varepsilon_a)L_p/N \right) + o(\log n) } \tag*{(using \eqref{eq:h} and $L_{p}\ll \log n$)} \nonumber \\ 
\label{eq:dV_II} &= e^{- (1+ \varepsilon_a) \log n + o(\log n)}, 
\end{align} 
where the last step uses \eqref{eq:LplogLp}.  Using \eqref{eq:dV_I}, \eqref{eq:dV_II}, and $1-(1-\theta_p)^{|V_{1, a}| } = (1+o(1)) |V_{1, a}|  \theta_p$ (since $|V_{1, a}|  \theta_p \leq n\theta_p \ll 1$, by Chernoff's inequality) gives, 
$$\P(d_{(V_{1, a})}(F_\kappa)\geq (1+ \varepsilon_a) L_p )  \geq (1+o(1)) |V_{1, a}| \theta_p  = 
e^{-  \varepsilon_a \log n + o(\log n)},$$ 
where the last step uses $|V_{1, a}| =\Theta_{\kappa, r}(n)$ and \eqref{eq:dV_II}.} 
\end{proof}

\subsection{Proof of Proposition \ref{ppn:deglt}}
\label{sec:degltpf}

We start with the trivial bound
\begin{align}\label{eq:degree_r_I}
\P(\cG_{n, p} \in \mathrm{degLT}_r(\bm \varepsilon)) \leq   \P(d_{(r)}(\cG_{n, p}) \leq (1-\varepsilon_r) L_p). 
\end{align} 
Note that $$\{d_{(r)}(\cG_{n, p}) \leq (1-\varepsilon_r) L_p\} \subseteq \left\{ S \subset [n]: |S|=r-1 \text{ and } \max_{v \in [n]\backslash S} d_{v}(\cG_{n, p}) \leq (1-\varepsilon_r) L_p \right\}.$$ Therefore, by \eqref{eq:degree_r_I} and a union bound, 
\begin{align}\label{eq:deglt_pf}
\P(\cG_{n, p} \in \mathrm{degLT}_r(\bm \varepsilon)) & \leq {n \choose r-1} \P\left(\max_{v \in \{ r , \ldots, n \} } d_{v}(\cG_{n, p}) \leq (1-\varepsilon_r) L_p \right)  \nonumber \\ 
& \leq n^r \P\left(\max_{v \in \{ r, \ldots, n \} } d_{v}(\cG_{n, p}) \leq (1-\varepsilon_r)  L_p \right). 
\end{align}
Now,  fix $\kappa \in (0, 1)$ and consider any $V_1 \subset \{r, \ldots, n\}$ with $|V_1|=n-\ceil{\kappa n}$. For each $v \in V_1$, denote by $d_v^+(\cG_{n, p})$ the number of edges from $v$ to the set $[n]\backslash V_1$. Clearly, 
$$\left\{ \max_{v \in \{ r, \ldots, n \} } d_{v}(\cG_{n, p}) \leq (1-\varepsilon_r)  L_p \right\} \subseteq \left\{ \max_{v \in V_1} d_v^+(\cG_{n, p}) \leq (1-\varepsilon_r) L_p \right\}.$$
This implies,  since $\{d_v^+(\cG_{n, p})  : v \in V_1\}$ is a collection of independent $\dBer(\ceil{\kappa n}, p)$ random variables, 
\begin{align*}
\P\left(\max_{v \in \{r, \ldots, n\}} d_{v}(\cG_{n, p}) \leq (1-\varepsilon_r) L_p \right) &\leq \P\left( \max_{v \in V_1} d_v^+(\cG_{n, p}) \leq (1-\varepsilon_r) L_p \right) \nonumber \\ 
& \leq \left(1- \P\left(\dBin(\ceil{\kappa n}, p) > (1-\varepsilon_r)  L_p \right) \right)^{n-\ceil{\kappa n}} \nonumber \\
& \leq \left(1- e^{- (1-\varepsilon_r )\log n +o(\log n) } \right)^{n-\ceil{\kappa n}},  
\end{align*}
by \eqref{eq:dV_II} (with  $1+\varepsilon_a$ replaced by $1-\varepsilon_r$). {This implies, from \eqref{eq:deglt_pf}, 
\begin{align*} 
\P(\cG_{n, p} \in \mathrm{degLT}_r(\bm \varepsilon))  \leq  n^r \left(1-\frac{1}{n^{1-\varepsilon_r +o(1)}}\right)^{n-\ceil{\kappa n}}  \leq  e^{-  (1+o(1)) (1 - \kappa) n^{ \varepsilon_r + o(1)} }, 
\end{align*}
which completes the proof of upper bound, by taking logarithm twice and limit as $n \rightarrow \infty$ followed by $\kappa \rightarrow 0$.} \\

For the lower bound $\P(\cG_{n, p} \in \mathrm{degLT}_r(\bm \varepsilon))$, fix a vertex $v \in [n]$ and note that 
\begin{align}
\P(\cG_{n, p} \in \mathrm{degLT}_r(\bm \varepsilon)) & \geq \P(d_v(\cG_{n, p}) \leq (1- \varepsilon_r) L_p)^n \tag*{(by FKG inequality, see Lemma \ref{lem:degfkg})} \nonumber\\
& \geq {\P( \dBin(n, p) \leq (1-\varepsilon_r) L_p)^n}  \nonumber \\
\label{eq:deg_r_pf}&\geq \left(1- e^{- n I_{p}(p+h_p(\varepsilon_r)) } \right)^n,  
\end{align}
by Chernoff's inequality, since $L_p \gg n p$, where $h_p(\varepsilon_r):=(1-\varepsilon_r) L_p/n-p$. Now, as in \eqref{eq:h} and \eqref{eq:LplogLp}, using $h_p(\varepsilon_r) \gg p$, it follows that 
$$n I_{p}(p+h_p(\varepsilon_r)) =(1+o(1)) n h_p(\varepsilon_r) \log (h_p(\varepsilon_r)/p)=(1-\varepsilon_r) \log n + o(\log n).$$ This implies, from \eqref{eq:deg_r_pf}, 
\begin{align*} 
\P(\cG_{n, p} \in \mathrm{degLT}_r(\bm \varepsilon))  \geq  \left(1-\frac{1}{n^{1-\varepsilon_r +o(1)}}\right)^n  \geq e^{- (1+o(1))n^{\varepsilon_r + o(1)}}. 
\end{align*}
This completes the proof of Proposition \ref{ppn:deglt}. \hfill $\Box$ \\

We finish the discussion with a justification of the first inequality in \eqref{eq:deg_r_pf} above. As mentioned, this is a consequence of the well-known FKG inequality for product measures between two decreasing events \cite[Chapter 2]{inequality}. In our context, an event $\cD \subseteq \sG_n$ is said to be {\it decreasing} if for two graphs $F_1=([n], E(F_1)), F_2=([n], E(F_2)) \in \sG_n$, with $E(F_2) \subset E(F_1)$, $F_1 \in \cD$ implies $F_2 \in \mathcal D$. Then the FKG inequality states that if $\cD_1, \cD_2 \subseteq \sG_n$ are two decreasing events, $\P(\cD_1| \cD_2) \geq \P(\cD_1)$.

\begin{lem}\label{lem:degfkg} For any $k>0$ fixed,
$$\P\left( d_{(1)}(\cG_{n, p}) \leq k \right) \geq \P( d_1(\cG_{n, p}) \leq k)^n.$$
\end{lem}
\begin{proof} 
Note that for every $1 \leq s \leq n$, the events $\{\max_{a \in [s]}d_{a}(\cG_{n, p}) \leq k \}$ and $\{d_s(\cG_{n, p}) \leq k\}$ are decreasing events. 
Therefore, iterated application of the FKG inequality gives, 
\begin{align*}
\P\left( d_{(1)}(\cG_{n, p}) \leq k \right) &\geq \P\left( \max_{a \in [n-1]}d_a(\cG_{n, p}) \leq k \right)\P( d_n(\cG_{n, p}) \leq k) \geq \P( d_1(\cG_{n, p}) \leq k)^{n},
\end{align*}
completing the proof of the lemma. 
\end{proof}

\section{Proof of Theorem \ref{thm:1ut}}
\label{sec:ut_pf} 

In this section we prove Theorem \ref{thm:1ut}. The section is organized as follows:   We start by recalling some elementary facts about the largest eigenvalue of a graph in Section  \ref{sec:degree}. 
In Section \ref{geom123} we prove some key structural results, including properties of sub-graphs induced by high degree vertices and the number of edge-disjoint cycles passing through various vertices. The proof of Theorem \ref{thm:1ut} is then presented in Section \ref{sec:ut_pf_I}. It has two cases: (1) where  $e^{-(\log \log n)^2} \leq n p \ll \sqrt{\log n/\log \log n}$ and (2) for  $(\log \log n)^2  \leq \log(1/n p)  \ll \log n$.

\subsection{Preliminaries}
\label{sec:degree}

We begin by recalling the following lemma from \cite{KS} which list some elementary properties of the maximum eigenvalue of a graph. 

\begin{lem}\label{lem:gfact}\cite[Proposition 3.1]{KS}
For any graph $F=(V(F), E(F))$ the following hold: 
\begin{enumerate}
\item[(a)] The maximum eigenvalue $\lambda_1(F)$ satisfies
$$\max\left\{\sqrt{d_{(1)}(F)}, \frac{2|E(F)|}{|V(F)|} \right\}\leq \lambda_1(F)\leq d_{(1)}(F),$$
where $d_{(1)}(F)$ is the maximum degree of $F$. 

\item[(b)] If $E(F) =\bigcup_{s=1}^K E(F_s)$, then $\lambda_1(F)\leq \sum_{s=1}^K \lambda_1(F_s)$. 

\item[(c)] If $F$ is a tree, then  $$\lambda_1(F)\leq \min\left\{2\sqrt{d_{(1)}(F)-1},\sqrt{|V(F)|-1} \right\}.$$ In particular, if $F$ is a star-graph with $d_{(1)}(F)+1$ vertices, then $\lambda_1(F)=\sqrt{d_{(1)}(F)}$.
\end{enumerate}
\end{lem}

\subsection{Geometric Properties}\label{geom123} As alluded to in Section \ref{sec:pf_outline}, the proof hinges on crucial geometric statements proved relying on estimates from \cite{KS}. To this end, let $X$ be the set of vertices in $\cG_{n, p}$ with degree larger than $np(1+1/\log\log n)+\Delta_p^{1/3}$, where $\Delta_p$ is defined in \eqref{eq:degreep}. Then define the following two events: 
\begin{align}\label{eq:A1}
\cA_1=\{\exists \text{ a cycle } C_L  \text{ in } \cG_{n, p}, \text{ with }  L \geq \log^{\frac{5}{6}}n,\text{ such that } |V(C_L) \cap X| \geq L/2 \}.
\end{align}
and
\begin{align}\label{eq:A2}
\cA_2=\{\exists~v \in [n] \text{ such that the number of neighbors of } v \text{ in } X \text{ is greater than } \Delta_p^{7/8}\}.
\end{align}
In the regime $p \geq e^{-(\log \log n)^2}/ n$, Krivelevich and Sudakov \cite[Lemma 2.3]{KS} showed, the probability of the event $\cA_2$ converges to zero. For the same range of $p$, they also showed that the event that $\cG_{n, p}$ has a cycle which intersects $X$ in more than half of its vertices has probability going to zero. Following their proof it easy to see that the event $\cA_2$ has probability converging to 0 at a super polynomial rate, that is, it is negligible in the upper tail large deviations scale. A straightforward modification of their proof also shows the same for the event $\cA_1$. (Note that in $\cA_1$ we only consider `long cycles' intersecting $X$ in more than half its vertices, unlike all cycles in \cite[Lemma 2.3]{KS}). This is summarized in the following lemma. The proof is given in Appendix \ref{sec:largecycle_pf}.\footnote{The choices of the exponents in the definitions of the events $\cA_1$ and $\cA_2$,  which might seem mysterious to the reader, are simply guided by their appearances in \cite{KS}. In fact, as will be evident from the proofs, there is significant slack in their choices. }

\begin{lem}\label{lem:largecycle} For $p \geq e^{-(\log \log n)^2}/n$  and $\cA_1$ and $\cA_2$ as defined above, 
$$\P\left(\cA_1 \bigcup \cA_2 \right)=  e^{-\Omega(\log^{\frac{17}{16}} n)}.$$ 
\end{lem} 

Next, we show, the probability that there are many vertex joint cycles passing through a vertex is also negligible. To this end, let $M=\log^{5/12} n$ and define $\cB$ to be the event that there exists a vertex in $\cG_{n, p}$ with $M$ edge disjoint cycles, each of size at most $\log^{5/6}n$, passing through that vertex. Also, denote by $\cB_0$ the event that there exists a vertex in $\cG_{n, p}$ with $M_0=L_{p}^{5/12}$ edge disjoint cycles passing through that vertex, without any restriction on the cycle size.
 
\begin{lem}\label{lem:blem}  For $n p \leq \log^{\frac{1}{2}} n$, $\P(\cB)=e^{-\Omega(\log^\frac{17}{12}n)}$. Moreover, for $np < 1$, $\P(\cB_0)=e^{-\omega(\log n)}$.
\end{lem}

\begin{proof} To begin with suppose $n p \leq \log^{1/2}n$. Note that the probability that a given vertex $v \in [n]$ has $M=\log^{5/12} n$ edge disjoint cycles passing through it, each of size at most $\log^{5/6}n$ is bounded by 
\begin{align*}
\left(\sum_{w=1}^{\log^{\frac{5}{6}}n} n^{w-1} p^{w} \right)^M & \leq \frac{1}{n^{M}} \left(\sum_{w=1}^{\log^{\frac{5}{6}}n} \log^{w/2}n \right)^M \tag*{(using $n p \leq \log^{\frac{1}{2}} n $)}\nonumber \\ 
& \lesssim \frac{(\log n)^{\frac{M}{2} \log^{\frac{5}{6}}n}}{n^{M}} =  \frac{(\log n)^{\frac{1}{2}\log^{\frac{15}{12}}n}}{n^{\log^{\frac{5}{12}} n}} = e^{-\Omega(\log^\frac{17}{12}n)} ,
\end{align*}
(To see the first term, note the number of $M$ tuples of edge disjoint cycles passing through a given vertex $v,$  of sizes $w_1,w_2,\ldots w_M,$ respectively, is at most $n^{\sum_{i=1}^M (w_i- 1)}$. The probability of any given tuple occurring in the graph $\cG_{n,p},$ owing to edge disjointness of the cycles is $p^{\sum_{i=1}^M w_i}.$ Therefore, summing over all possible values of $w_1,w_2,\ldots w_M,$ between $1$ and $M$ yields the first term.) 
Then, by a union bound over the choice of the vertex $v\in [n]$, it follows that $\P(\cB)=n e^{-\Omega(\log^\frac{17}{12}n)} =e^{-\Omega(\log^\frac{17}{12}n)}$. 

Next, suppose $n p:=c < 1$. Then, as above, the probability that a given vertex $v \in [n]$ has $M_0=L_p^{5/12}$ edge disjoint cycles (each has size trivially at most $n$) passing through it is at most 
\begin{align*}
\left(\sum_{w=1}^{n} n^{w-1} p^{w} \right)^{M_0} & \leq \frac{1}{n^{M_0}} \left(\sum_{w=1}^{n} c^{w} \right)^{M_0}  \lesssim \frac{1}{n^{M_0}} = e^{-\omega(\log n)}.
\end{align*} 
since $L_p\gg 1.$
Hence, by a union bound $\P(\cB_0)=n e^{-\omega(\log n)} =e^{-\omega(\log n)}$. 
\end{proof}

Next, denote by $\cC$ the event that $\cG_{n, p}$ contains a vertex which has at least $\Delta_p^{1/3}$ other vertices of $\cG_{n, p}$, each with degree greater than $\Delta_p^{3/4}$, within distance at most two. The proof of \cite[Lemma 2.4]{KS} shows that the probability $\cC$ is also negligible in the upper tail large deviations scale. 

\begin{lem}\label{lem:clem}\cite[Lemma 2.4]{KS}   For $e^{-(\log \log n)^2} \leq n p \leq \log ^{\frac{1}{2}} n $, $\P(\cC)=e^{-\Omega(\log^{\frac{25}{24}}n)}$. 
\end{lem}

The exponents $17/16, 17/12$, and $25/24$ obtained in the above lemmas are useful only to the extent that they are bigger than one, rendering the above events much more unlikely compared to the upper tail event in Theorem \ref{thm:1ut}. 
Therefore, to compute the upper tail probability it suffices to restrict ourselves to the complement of the union of the above events. To this end, denoting $\cA=\cA_1\bigcup \cA_2$, we define  
\begin{equation}\label{eq:utgood}
\cM :=\overline \cA \bigcap \overline \cB \bigcap \overline \cC,
\end{equation}
where $\overline S$ denotes the complement of a set $S$. In other words, the event $\cM$ is the collection of graphs $F \in \sG_n$ satisfying the following four properties: 
\begin{enumerate}
\item There is no cycle of length $L \geq \log^{5/6}n$ in $F$ which intersects $X$ in at least $L/2$ vertices. 
\item The maximum degree of each vertex of $F$ in $X$ is at most $\Delta_p^{7/8}$. 
\item For every vertex in $F$ there is at most $\log^{5/12}n$ edge disjoint cycles, each of size at most $\log^{\frac{5}{6}}n$, passing through that vertex. 
\item There is no vertex which has at least $\Delta_p^{1/3}$ other vertices of $F$, each with degree greater than $\Delta_p^{3/4}$, within distance at most two.
\end{enumerate}

\subsection{Proof of Theorem \ref{thm:1ut}}
\label{sec:ut_pf_I}

We can now use the properties above to prove our main structure theorem, which shows that any graph $F \in \cM$ can decomposed into two components, one consisting of a disjoint union of stars and another component which has negligible spectral norm.\footnote{For two graphs $F_1=([n], E(F_1))$ and $F_2=([n], E(F_2))$ on $[n]:=\{1, 2, \ldots, n\}$ vertices, denote by $F_1\bigcup F_2=([n], E(F_1)\bigcup E(F_2))$, the graph obtained by taking the union of the edges in the two graphs.} The idea behind the proof is to partition  the vertices of $F$ based on degrees, and then iteratively remove cycles incident on the `high' degree vertices to decompose $F$ in to a forest and another component which has spectral norm $o(\sqrt{L_p})$  (which can be ensured by Lemma \ref{lem:largecycle} and \ref{lem:blem}). Then Lemma \ref{lem:clem} can be used to argue that the component of the forest incident on the high degree vertices must be a disjoint union of stars.  The formal proof is presented in two cases: (1) $e^{-(\log \log n)^2} \leq n p \ll \sqrt{\log n/\log \log n}$, and (2)  $(\log \log n)^2  \leq \log(1/n p)  \ll \log n$. We begin with the first case.\\

\noindent {\it Case} 1: $e^{-(\log \log n)^2} \leq n p \ll \sqrt{\log n/\log \log n}$. The following lemma is our key structural result formalizing the above discussion.
\begin{lem}\label{lem:decompose} Suppose  $e^{-(\log \log n)^2} \leq n p \ll \sqrt{\log n/\log \log n}$. Then any graph $F \in \cM$ can be decomposed as $F=F_1 \bigcup F_2$ such that $F_1$ is a disjoint union of stars, the spectral norm $||F_2||= o(\sqrt{L_p})$, and maximum degree $d_{(1)}(F_2)=o(L_p)$. \end{lem}

The proof of this lemma is given below. We first show how this can be used to complete the proof of Theorem \ref{thm:1ut}. Throughout the proof we denote $\overline {\bm \delta}=(\bm 1+\bm \delta)^2-\bm 1$, where $\bm 1=(1, 1, \ldots, 1)$ is a vector of length $r$ with all entries 1. \\ 

We start with the lower bound. 
Using $\P(\cG_{n, p} \in \cM)=1-e^{-\omega(\log n)}$ (which follows by combining Lemma \ref{lem:largecycle}, \ref{lem:blem}, and \ref{lem:clem}) and Proposition \ref{ppn:degree_ut} (which gives $\P(\cG_{n, p} \in \mathrm{degUT}_r(\overline {\bm \delta}))=e^{-\Theta(\log n)}$) note that, 
\begin{equation}\label{eq:utMlb}
\P(\cG_{n, p} \in \cM|\cG_{n, p} \in \mathrm{degUT}_r(\overline {\bm \delta})) \rightarrow 1. 
\end{equation}
Now, by Lemma \ref{lem:decompose}, if $F \in \cM \bigcap \mathrm{degUT}_r(\overline {\bm \delta})$, then $F=F_1\bigcup F_2$, where 
\begin{itemize}
\item  $F_2$ satisfies $\lambda_1(F_2)= o(\sqrt{L_p})$ and $d_{(1)}(F_2)=o(L_p)$, and 
\item $F_1$ is a {disjoint union of $\ell$ star graphs} isomorphic to $K_{1, s_1}, K_{1, s_2}, \ldots, K_{1, s_\ell}$, for some $s_1 \geq s_2 \geq \ldots \geq s_\ell $ and $\ell \geq r$. Moreover, $s_a \geq ((1+\delta_a)^2-o(1)) L_p$, for $1 \leq a \leq r$.   
\end{itemize}
The second conclusion follows from the fact that $\cG_{n, p} \in \mathrm{degUT}_r(\overline {\bm \delta})$ while $d_{(1)}(F_2)=o(L_p).$
Therefore, by Lemma \ref{lem:gfact}(c) $\lambda_a(F_1) = \sqrt{s_a} \geq (1+\delta_a+o(1)) \sqrt{L_p}$, and by Weyl's inequality,  
$$\lambda_a(F) \geq \lambda_a(F_1)+\lambda_n(F_2) \geq (1+\delta_a+o(1)) \sqrt{L_p} + o(\sqrt{L_p})=(1+\delta_a + o(1)) \sqrt{L_p},$$
for all $1 \leq a \leq r$. This shows, by adjusting $\overline{\bm \delta}$ by an $o(1)$ term, $$\cM \bigcap \mathrm{degUT}_r((1+o(1)) \overline{\bm \delta}) \subseteq \mathrm{UT}_r( \bm \delta).$$
Hence,
\begin{align}\label{eq:ub_M}
\P(\cG_{n, p} \in \mathrm{UT}_r(\bm \delta )) & \geq \P\left( \cG_{n, p} \in \cM \bigcap \mathrm{degUT}_r( (1+o(1))\overline{\bm \delta})\right) \nonumber \\ 
& = \P(\cG_{n, p} \in \cM| \cG_{n, p} \in \mathrm{degUT}_r( (1+o(1)) \overline{\bm \delta}))  \P(\cG_{n, p} \in \mathrm{degUT}_r( (1+o(1)) \overline{\bm \delta})) \nonumber \\ 
& =(1+o(1))\P(\cG_{n, p} \in \mathrm{degUT}_r( (1+o(1)) \overline{\bm \delta})),
\end{align}
where the last step uses \eqref{eq:utMlb}. 

To obtain a matching upper bound on $\P(\cG_{n, p} \in \mathrm{UT}_r(\bm \delta ))$, note that \eqref{eq:ub_M} along with Proposition \ref{ppn:degree_ut}  shows that $\P(\cG_{n, p} \in \mathrm{UT}_r(\bm \delta)) \geq e^{-\Theta(\log n)}$, which implies, as in \eqref{eq:utMlb}, 
\begin{align}\label{eq:utMub}
\P(\cG_{n, p} \in \cM| \cG_{n, p} \in \mathrm{UT}_r(\bm \delta)) \rightarrow 1. 
\end{align} 
{Moreover, if $F \in \cM  \bigcap \mathrm{UT}_r(\bm \delta)$, then by Lemma \ref{lem:decompose},  $F=F_1\bigcup F_2$, where $F_2$ satisfies $\|F_2\|= o(\sqrt{L_p})$ and $d_{(1)}(F_2)=o(L_p)$, and $F_1$ is a disjoint union of $\ell$ star graphs isomorphic to $K_{1, s_1}, K_{1, s_2}, \ldots, K_{1, s_\ell}$, for some $s_1 \geq s_2 \geq \ldots \geq s_\ell $ and $\ell \geq 0$. Since the spectrum of $F_1$ is the disjoint union of the spectra of its connected components and $\|F_2\|=o(\sqrt{L_p}),$ a simple application of Weyl's inequality along with $F \in \cM  \bigcap \mathrm{UT}_r(\bm \delta)$ implies that
\begin{align}\label{eq:lambdaF1}
\lambda_{a}(F_1) &\geq \lambda_a(F) + o(\sqrt{L_p}) \geq (1+\delta_a)\sqrt{L_p}+o(\sqrt{L_p}).
\end{align} 
Moreover, since the spectrum of a $s$-star $K_{1,s}$ has only  two non-zero eigenvalues $-\sqrt s$ and $\sqrt s$, by \eqref{eq:lambdaF1}, 
$$d_{(a)}(F) \geq d_a(F_1) + o(L_p) \geq (1+\delta_a)^2 L_p + o(L_p),$$ for $1 \leq a \leq r$. Thus, it follows that $F \in \mathrm{degUT}_r(1+o(1)) \overline{\bm \delta})$. Therefore, $\cM  \bigcap \mathrm{UT}_r(\bm \delta) \subseteq  \mathrm{degUT}_r((1+o(1)) \overline{\bm \delta})$, and 
\begin{align}\label{eq:lb_M}
\P(\cG_{n, p} \in \mathrm{UT}_r(\bm \delta)) & = (1+o(1)) \P\left(\cG_{n, p} \in  \cM  \bigcap \mathrm{UT}_r(\bm \delta) \right) \nonumber \\ 
& \leq  (1+o(1))\P(\cG_{n, p} \in \mathrm{degUT}_r((1+o(1)) \overline{\bm \delta})), 
\end{align}
where the first equality uses \eqref{eq:utMub}. The claim in Theorem \ref{thm:1ut} now follows by combining \eqref{eq:ub_M} and \eqref{eq:lb_M} with Proposition \ref{ppn:degree_ut}.  
We finish the proof in Case 1, with the details of the proof of Lemma \ref{lem:decompose}.
\subsubsection*{\textbf{Proof of Lemma \ref{lem:decompose}}} We begin by partitioning the vertices of a graph $F \in \cM$ in to three components $[n]=X_1 \bigcup X_2 \bigcup Y$ as follows: 
\begin{itemize}
\item $X_1$ is the set of vertices in $F$ with degree at least $\Delta^{3/4}_p$ (the `high' degree vertices). 

\item $X_2$ is the set of vertices in degree between  $np(1+1/\log\log n)+\Delta^{1/3}_p$ and $\Delta^{3/4}_p$ (the `moderate' degree vertices). Recall, from the discussion above \eqref{eq:A1}, that $X=X_1 \bigcup X_2$. 

\item $Y=[n]\backslash X$ denotes the remaining set of vertices (the `low' degree vertices). 

\end{itemize}
Let $G_1$ be the induced subgraph of $F$ on the set $X$, $G_2$ the induced subgraph of $F$ on the set $Y$, $G_3$ the bipartite subgraph of $F$ with edges between $X_1$ and $Y$, and $G_4$ the the bipartite subgraph of $F$ with edges $X_2$ and $Y$,  respectively (as shown in  Figure \ref{fig:graph_decomposition}).

\begin{figure}[h]
\centering
\begin{minipage}[l]{1.0\textwidth}
\centering
\includegraphics[width=4.95in]
    {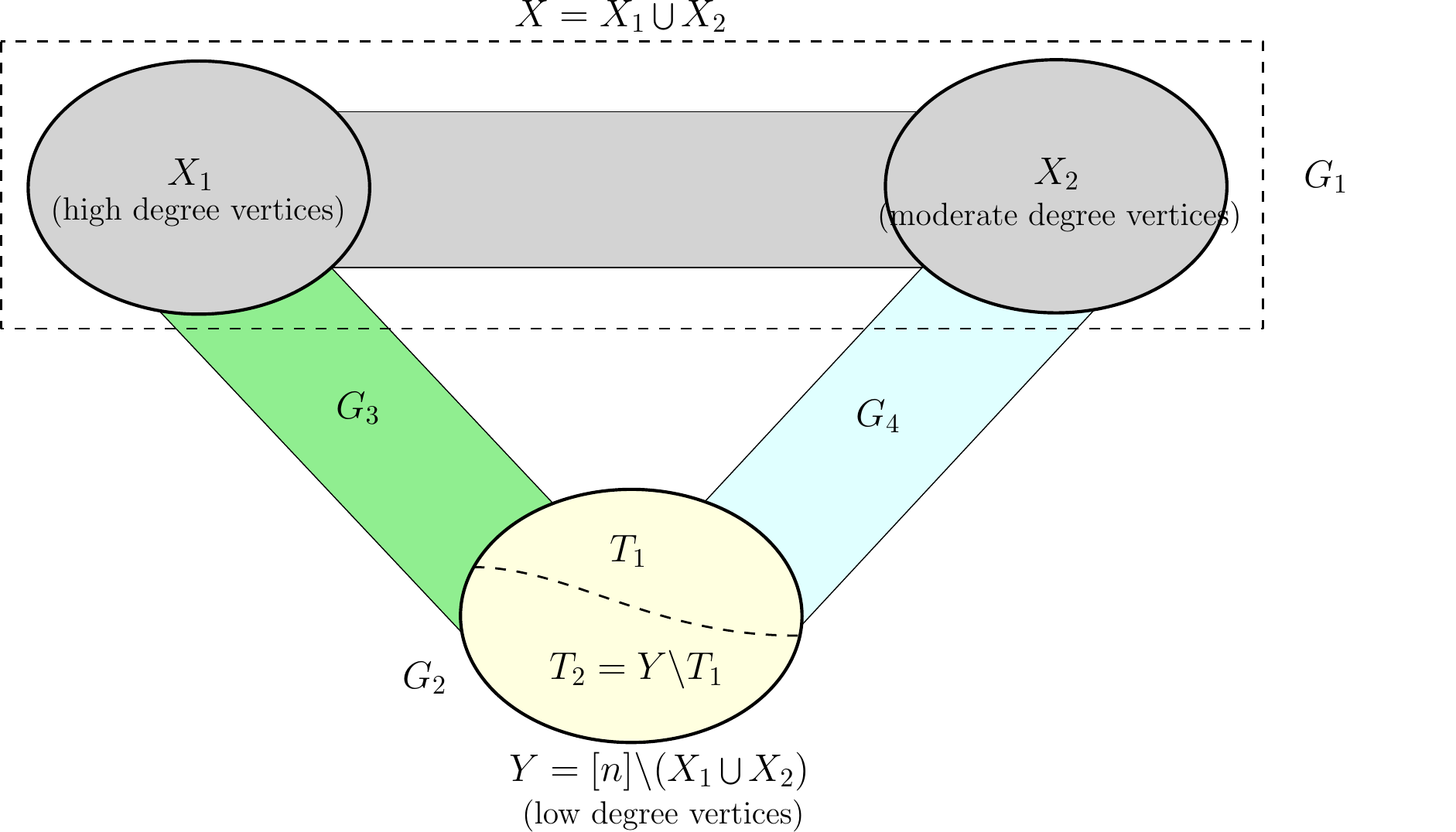}\\
\end{minipage}
\caption{\small{The partition of the vertices and the edges of the graph $F$ as in the proof of Lemma \ref{lem:decompose}. Here, $T_1$ denotes the set of vertices in $Y$ with degree 1 in $X_1$ after removing cycles from the graph $G_3$.}}
\label{fig:graph_decomposition}
\end{figure}

Now, by recalling the current regime of $p$ and Lemma \ref{equivalence},  
\begin{align}\label{eq:NN}
np \ll \sqrt{L_p}\overset{\eqref{relation345}}=(1+o(1)) \sqrt{\Delta_p} \quad \text{ and } \quad L_p\ge \frac{\log n}{(\log\log n)^2}.
\end{align}
Thus $d_{(1)}(G_2)\le np(1+1/\log\log n)+\Delta^{1/3}_p=o(\sqrt{L_p}).$ Applying Lemma \ref{lem:gfact}  this implies  
\begin{align}\label{eq:G2}
||G_2|| \leq d_{(1)}(G_2)=o(\sqrt{L_p}). 
\end{align}
We now iteratively remove cycles from $G_1$, in arbitrary order, until $G_1$ is cycle free.  Denote by $H_1$ the graph obtained by the union of the edges removed from $G_1$. (Note that, by construction, the edges of $H_1$ can be partitioned into disjoint union of cycles.) Then iteratively remove cycles from $G_4$ until $G_4$ is cycle free, and denote the removed graph by $H_4$. Finally, iteratively remove cycles from $G_3$ until $G_3$ is cycle free, and denote the removed graph by $H_3$. Observe that each vertex in every cycle removed from $G_1$ has degree greater than $\Delta_p^{1/3}$ in $F$. Moreover, because $G_3$ and $G_4$ are bipartite graphs, at least half of the vertices in every cycle removed from $G_3$ and $G_4$ has degree greater than $\Delta_p^{1/3}$ in $F$. Therefore, by  property (1) of the set $\cM$ (listed right after \eqref{eq:utgood}) there are no cycles of length greater than $\log^{5/6} n$ in $G_1$, $G_3$, or $G_4$. Then by  property (3) of the set $\cM$, 
$$\max\left\{d_{(1)}(H_1), d_{(1)}(H_3), d_{(1)}(H_4) \right\} = O(\log^{5/12} n) = o(\sqrt{L_p}),$$ 
where the final conclusion follows from the lower bound on $L_p$ in \eqref{eq:NN}. This implies, 
\begin{align} 
||G_4||  \leq  ||G_4 \backslash H_4|| + ||H_4|| \nonumber & \leq 2 \sqrt{d_{(1)}(G_4)} + d_{(1)}(H_4) \tag*{(using Lemma \ref{lem:gfact})}\nonumber \\
\label{eq:G4} &=O(L_p^{3/8})+o(\sqrt{L_p})=o(\sqrt{L_p}),
\end{align}
where the second inequality uses $ ||G_4 \backslash H_4||=\lambda_1(G_4 \backslash H_4)$, and  Lemma \ref{lem:gfact}(c), since $G_4 \backslash H_4$ is a forest. Similarly, since $G_1 \backslash H_1$ is a forest, using property (2) of $\cM,$
\begin{align}\label{eq:G1}
||G_1|| \leq ||G_1 \backslash H_1|| + ||H_1|| & \leq  2 \sqrt{d_{(1)}(G_1)} + d_{(1)}(H_1) \nonumber \\ 
& \leq O(\Delta_p^{7/16}) + O(\log^{5/12} n)= o(\sqrt{L_p}).
\end{align}
Now, partition the vertices of $Y$ into two parts $Y_1=T_1\cup T_2$, where $T_1$ is the set of the vertices in $Y$ with degree 1 in the graph $G_3\backslash H_3$, and $T_2= Y \backslash T_1$. Note that for every $v \in X_1$, the number of neighbors of $v$ in $T_2$ in the graph  $G_3\backslash H_3$ is at most $\Delta_p^{1/3}.$ (To see this, list the neighbors of $v$ in $T_2$ in the graph  $G_3\backslash H_3$ as $u_1,u_2,\ldots, u_K$, for some $K \geq 1$. Since $u_i \in T_2$ each of them have at least another neighbor  $v_i  \in X_1,$ for $1 \leq i \leq K$. Now,  because $G_3\backslash H_3$ is a forest all the vertices $v_1, v_2, \ldots, v_K$ must be distinct and at distance two from $v.$ Thus, $K \le \Delta_p^{1/3}$  by property (4) of the set $\cM$.)  Then we split the graph $G_3\backslash H_3$ in to two parts $F_1$ and $G_{3, 2}$, where $F_1$ is the set of all edges from $X_1$ to $T_1$ and $G_{3, 2}$ is the set of edges from $X_1$ to $T_2$. Note that the maximum degree of the graph $G_{3, 2}$ is $np(1+1/\log\log n)+\Delta^{1/3}_p=o(\sqrt{L_p}).$ This is because the maximum degree of any vertex in $Y$ is at most this and so is the maximum degree of any vertex in $X_1$ in to $T_2$ by the above discussion. Hence 
\begin{align}\label{eq:G3_II}
||G_{3, 2}|| \leq d_{(1)}(G_{3, 2})=o(\sqrt{L_p}).
\end{align}
Therefore, setting $F_2=G_1\bigcup G_2 \bigcup G_4\bigcup G_{3, 2}$, and combining \eqref{eq:G1}, \eqref{eq:G2},  \eqref{eq:G4}, and \eqref{eq:G3_II}, and Lemma \ref{lem:gfact}(b), it follows that $||F_2||=o(\sqrt{L_p})$.  This completes the proof since, by construction, $F_1$ is a disjoint union of stars. \hfill $\Box$  \\

\noindent {\it Case} 2: $(\log \log n)^2  \leq \log(1/n p)  \ll \log n$. First we prove the lower bound on the upper tail event $\mathrm{UT}_r(\bm \delta)$.  
As in the previous case, let $\overline {\bm \delta}=(\bm 1+\bm \delta)^2-\bm 1$, where $\bm 1=(1, 1, \ldots, 1)$ is the all 1 vector of length $r$. In this case, recall from Lemma \ref{lem:blem} that $\P(\overline{\cB_0}) = 1-e^{-\omega(\log n)}$. This implies, by Proposition \ref{ppn:degree_ut} (which gives $\P(\cG_{n, p} \in \mathrm{degUT}_r(\overline {\bm \delta}))=e^{-\Theta(\log n)}$), 
$$\P(\cG_{n, p} \in \overline{\cB_0}|\cG_{n, p} \in \mathrm{degUT}_r(\overline {\bm \delta})) \rightarrow 1.$$ Then, $\P(\cG_{n, p} \in \mathrm{degUT}_r(\overline{\bm \delta} ) ) = (1+o(1)) \P(\cG_{n, p} \in \mathrm{degUT}_r(\overline{\bm \delta} ) \bigcap \overline{\cB_0} ) $, and, by \eqref{eq:degree_ut},
\begin{align}\label{eq:deg_B}
\lim_{n \rightarrow \infty} \frac{-\log \P(\cG_{n, p} \in \mathrm{degUT}_r(\overline{\bm \delta} ) \bigcap \overline{\cB_0} )}{\log n}=  \sum_{a=1}^{r}(2\delta_a+\delta^2_a).  
\end{align}
Now, suppose $G \sim \mathrm{degUT}_r(\overline{\bm \delta}) \cap \overline{\cB_0}$, and iteratively remove cycles from $G$, in arbitrary order, until $G$ is cycle free.  Denote by $H$ the graph obtained by the union of the edges removed from $G$, and $G_1=G\backslash H$. Note that, by construction, the edges of $H$ can be partitioned into disjoint union of cycles. Moreover, by definition of the set $\cB_0$, no vertex of $G$ has more than $M=L_p^{5/12} $ edge disjoint cycles passing through it.  Therefore, $d_{(1)}(H) = O(L_p^{5/12}) = o(\sqrt{L_p})$. 
This implies, for $1 \leq a \leq r$, 
\begin{align}\label{eq:degreeG1}
d_{(a)}(G_1) \geq d_{(a)}(G)  - d_{(1)}(H) \geq (1+ \overline \delta_a) L_p - o(L_p)
\end{align}
that is, $G_1\in \mathrm{degUT}_r((1+o(1)) \overline{\bm \delta}).$  Further, by definition, $G_1$ is a disjoint union of the trees.  To proceed we need an estimate on the sizes of the tree components. {To this end, for $\theta>0$, denote by $\cD_{\theta}$ the event that $\cG_{n, p}$ has a tree of size 
$\ceil{\theta L_p}+1$. Next, for $s \geq 1$ and $\theta_1, \theta_2, \ldots, \theta_s > 0$, let $\cD_{\theta_1, \theta_2, \ldots, \theta_s}$ be the event that $\cG_{n, p}$ has $s$ disjoint trees of sizes  $\left\lceil \theta_1 L_p  \right\rceil +1 , \left\lceil \theta_2 L_p \right \rceil +1, \ldots,  \left\lceil \theta_s L_p  \right \rceil +1$, respectively} (note that they do not have to be connected components of $\cG_{n,p}$). We now bound the probability of $\cD_{\theta_1, \theta_2, \ldots, \theta_s}$.

\begin{lem}\label{lem:dlem} Suppose $(\log \log n)^2  \leq \log(1/n p)  \ll \log n$. Then for $s \geq 1$ and $\theta_1, \theta_2, \ldots, \theta_s > 0$, 
\begin{align*}
\P(\cD_{\theta_1, \theta_2, \ldots, \theta_s})  \leq \exp\left\{- \left( \sum_{a=1}^s\theta_a - s  \right) \log n  + o(\log n) \right\}.
\end{align*}
Moreover, $\P(\cD_{\log n}) \leq e^{-\Theta(\log^2 n)}$. 
\end{lem}

\begin{proof} For $\theta >0$, denote by $T_{n, \theta}$ the number of labeled trees of size $t=\ceil{ \theta L_p} +1 $ that can be formed with $n$ vertices. Note that in the regime  $(\log \log n)^2  \leq \log(1/n p)  \ll \log n$, $L_p= (1+o(1)) \frac{\log n}{\log(1/np)}$. Then, using the fact that the number of labelled trees on $t$ vertices is $t^{t-2}$, gives
\begin{align}
\P(\cD_{\theta}) = \P(T_{n, \theta} > 0) \leq \E(T_{n, \theta})  & = {n \choose t} t^{t-2} p^{t-1} \nonumber \\ 
& \leq \frac{n^t}{t!} t^{t-2} p^{t-1} \nonumber \\ 
& \leq \frac{en}{t^2} (enp)^{t-1}  \leq e^{t - (t-1)\log(1/np) + \log n  } \tag*{(since $t > 1$)}\nonumber \\ 
& {\leq \exp\left\{ O_{\theta}\left(\frac{\log n}{\log(1/np)}\right)  -  ( \theta -1 ) \log n +O(\log (1/np))  \right \} } 
\nonumber \\ 
& \leq \exp\left\{- \left( \theta - 1  \right) \log n  + o(\log n) \right\}, \nonumber 
\end{align}
where the last step uses both ${\log n}/{\log(1/np)}$ and $\log(1/np)$ are $o(\log n)$, which follows by our hypothesis on $p$ in the current regime. Now, taking 
$\theta=\log n$, it follows that $\P(\cD_{\log n}) \leq e^{-\Theta(\log^2 n)}$.

{Next, for $\theta_1, \theta_2, \ldots, \theta_s > 0$, since the event $\cD_{\theta_1, \theta_2, \ldots, \theta_s}$ demands the disjoint occurrence of the events $\cD_{\theta_1}, \ldots , \cD_{\theta_s},$ we can apply the Van-den Berg-Kesten inequality (BK) inequality \cite{inequality_II} to obtain
\begin{align*}
\P(\cD_{\theta_1, \theta_2, \ldots, \theta_s})  \le  \prod_{a=1}^s \P(\cD_{\theta_a})  \leq \exp\left\{- \left( \sum_{a=1}^s\theta_a - s  \right) \log n  + o(\log n) \right\},
\end{align*} 
completing the proof of the lemma.}
\end{proof}

We will now analyze the event $G_1\in \mathrm{degUT}_r((1+o(1))\overline{\bm \delta}).$ 
Note that $\mathrm{degUT}_r((1+o(1))\overline{\bm \delta})$ entails the existence of vertices $v_1,v_2,\ldots, v_r \in V(G_1)$, with $d_{v_a}(G_1)\ge (1+o(1))(1+ \overline \delta_a )L_p$, for $1 \leq a \leq r$. Denoting the connected tree components of $G_1$ as $G_{1, 1}, G_{1, 2}, \ldots, G_{1, \nu}$, for some $\nu \geq 1$, let $\eta_a \in \{1, 2, \ldots, \nu\}$, be the index of the tree component containing the vertex $v_a$,  for $1 \leq a \leq r$. The following lemma shows, on the event  $G_1\in \mathrm{degUT}_r((1+o(1))\overline{\bm \delta})$, it is overwhelming likely that $\eta_a \neq \eta_b$, for all $1\le a\neq b \le r$.  

\begin{lem}  $\P\left(\{\eta_a \ne \eta_b \text{ for all } 1\le a \neq b \le r\}| \cG_{n, p} \in \mathrm{degUT}_r(\overline {\bm \delta}) \bigcap \overline{\cB_0} \right) \rightarrow 1$. 
\end{lem}

\begin{proof} 
Throughout this proof we order the connected components $G_{1,1}, G_{1,2}, \ldots, G_{1, \nu}$ of $G_1$,  such that $|V(G_{1,1})| \geq |V(G_{1,2})| \geq \cdots \geq |V(G_{1, \nu})|$. To proceed with the proof we develop some notation which allows us to encode which high degree vertices come from the same tree component.

To this end, fix a partition $\Pi$  of the set $[r]=\{1, 2, \ldots, r\}$ with parts $\Pi=\{\pi_1, \pi_2, \ldots, \pi_B\}$, for some $B \geq 1$. (Note that $\pi_1, \pi_2, \ldots, \pi_B$ are mutually disjoint subsets of $[r]$ and $\sum_{b=1}^B |\pi_b|=r$.) Given a partition $\Pi$ as above, denote by $\mathrm{deg}_\Pi$ the event that, for all $1 \leq b \leq B$, 
\begin{align}\label{eq:degree_pi}
\eta_{j_1}=\eta_{j_2}, \quad \text{ for all } \quad j_1, j_2 \in \pi_b, 
\end{align}
that is, all the vertices in $\pi_b$ belong to the same tree component. In other words, if we denote the common value of $\eta$ as $\phi(b)$,  then \eqref{eq:degree_pi} can be rewritten as $\{v_{a}: a \in \pi_b\} \subseteq V(G_{1, \phi(b)}),$  where $\phi: \{1, 2, \ldots, B\} \rightarrow  \{1, 2, \ldots, \nu\} $ is some injective function.
{Now, note that, since $G_1$ is a forest, no two vertices in $G_1$ can have more than 1 common neighbor (otherwise the graph $G_1$ will have a cycle). Then, on the event $\mathrm{deg}_\Pi\bigcap \left( \mathrm{degUT}(\overline{\bm \delta}) \bigcap \overline{\cB_0}  \right)$, for $1 \leq b \leq B$,  we have the following bound on the size of $V(G_{1, \phi(b)}),$
\begin{align*}
|V(G_{1, \phi(b)})| \geq \sum_{a \in \pi_b} d_{v_a}(G_1) - O(r^2) \geq \sum_{a \in \pi_b} (1+\overline \delta_a) (1+o(1)) L_p = \sum_{a \in \pi_b} (1+\delta_a)^2  (1+o(1)) L_p. 
\end{align*} 
Above the first inequality uses the mentioned fact that no two vertices have more than one common neighbor, and the second inequality uses the above mentioned lower bound on $d_{v_a}(G_1)$.  This shows $|V(G_{1, \phi(b)})| \geq \sum_{a \in \pi_b} (1+\delta_a)^2  (1+o(1)) L_p + {1}$ (The additional plus $1$ which can be achieved by adjusting the $o(1)$ term is written to help invoke Lemma \ref{lem:dlem} verbatim).} For notational brevity, denote $w_\Pi(b)= (1+o(1)) \sum_{a \in \pi_b} (1+\delta_a)^2$. Therefore, $\mathrm{deg}_\Pi\bigcap \left( \mathrm{degUT}(\overline{\bm \delta}) \bigcap \overline{\cB_0}  \right)$ implies the event $\cD_{w_\Pi(1), \ldots, w_\Pi(B)}$. This implies, by Lemma \ref{lem:dlem}, 
\begin{align}\label{eq:deg_partiton}
\P\left(\cG_{n, p} \in  \mathrm{deg}_\Pi\bigcap \left( \mathrm{degUT}(\overline{\bm \delta}) \bigcap \overline{\cB_0}  \right)  \right) & \leq \P(\cD_{w_\Pi(1), \ldots, w_\Pi(B)})  \nonumber \\ 
&  \leq \exp\left\{- \left( \sum_{b=1}^B w_\Pi(b) - B  \right) \log n  + o(\log n) \right\} \nonumber \\ 
&  \leq \exp\left\{- \left(   \sum_{a=1}^r (1+\delta_a)^2 - B  \right) \log n  + o(\log n) \right\}, 
\end{align} 
where the last step uses $\sum_{b=1}^B w_\Pi(b) = (1+o(1)) \sum_{a=1}^r (1+\delta_a)^2$. Now, let $\cP_{< r}$ be the set of all partitions of $\{1, 2, \ldots, r\}$ with strictly less than $r$ parts. Then, using a union bound and $|\cP_{< r}| \lesssim_r 1$.  
\begin{align*}
\P& \left(\{\eta_a = \eta_b \text{ for some } 1\le a \neq b \le r\}| \cG_{n, p} \in \mathrm{degUT}_r(\overline {\bm \delta}) \bigcap \overline{\cB_0} \right) \nonumber \\   
  &  ~~~~~~~~~~~~~~~~ \leq \frac{\sum_{\Pi  \in \cP_{< r}}  \P\left(\cG_{n, p} \in  \mathrm{deg}_\Pi\bigcap \left( \mathrm{degUT}(\overline{\bm \delta}) \bigcap \overline{\cB_0}  \right)  \right)}{\P(\cG_{n, p} \in \mathrm{degUT}_r(\overline {\bm \delta}) \bigcap \overline{\cB_0})} \nonumber \\ 
&   \lesssim_r \frac{ e^{- \left(   \sum_{a=1}^r (1+\delta_a)^2 - (r-1)  \right) \log n  + o(\log n) } }{e^{- \left(   \sum_{a=1}^r (1+\delta_a)^2 - r  \right) \log n  + o(\log n) } } \tag*{(using \eqref{eq:deg_B} and $|B| \leq r-1$ in \eqref{eq:deg_partiton})} \nonumber \\ 
& \rightarrow 0. 
\end{align*}
This completes the proof of the lemma. 
\end{proof}

The lemma above implies 
$$\P\left(\cG_{n, p} \in \mathrm{degUT}_r(\overline {\bm \delta}) \bigcap \overline{\cB_0} \right) = (1+o(1)) \P\left(\{\eta_a \ne \eta_b \text{ for all } 1\le a \neq b \le r\} \bigcap  \mathrm{degUT}_r(\overline {\bm \delta}) \bigcap \overline{\cB_0} \right).$$ 
Thus, \eqref{eq:deg_B} gives,  
$$\lim_{n \rightarrow \infty} \frac{-\log \P\left(\cG_{n, p} \in \mathrm{degUT}_r((1+o(1))\overline{\bm \delta})\bigcap \overline{\cB_0} \bigcap \{ \eta_a \ne \eta_b \text{ for all } 1\le a \neq b \le r\}\right)}{\log n}=\sum_{a=1}^{r}(2 \delta_a + \delta_a^2).$$
Then, by Lemma \ref{lem:deg_eta} below, 
\begin{align}\label{eq:ut_lb_B}
\P(\cG_{n, p} \in \mathrm{UT}_r(\bm \delta)) \geq \exp\left\{- (1+o(1)) \sum_{a=1}^r (\delta_a^2 + 2 \delta_a)  \log n  \right\}, 
\end{align}
which proves the lower bound in the regime $(\log \log n)^2  \leq \log(1/n p)  \ll \log n$. 

\begin{lem}\label{lem:deg_eta}The event $\mathrm{degUT}_r((1+o(1))\overline{\bm \delta})\bigcap \overline{\cB_0} \bigcap \{ \eta_a \ne \eta_b \text{ for all } 1\le  a \neq b \le r\}$  implies the upper tail event $\mathrm{UT}_r( (1+o(1)) \overline{\bm \delta})$. 
\end{lem}

\begin{proof}
Recall that $G_1$ is a disjoint union of trees. Therefore, on the event $$\mathrm{degUT}_r((1+o(1))\overline{\bm \delta})\bigcap \overline{\cB_0} \bigcap \{ \eta_a \ne \eta_b \text{ for all } 1\le  a \neq b \le r\},$$ $G_1$ has at least $r$ connected components $G_{1, 1}, G_{1, 2}, \ldots, G_{1, r}$, such that, for $1 \leq a \le r$, there exists a vertex $v_a \in V(G_{1,  a})$ with $d_{v_a}(G_{1,  a})\ge (1+o(1))(1+ \overline \delta_a)L_p$. This implies, by Lemma \ref{lem:gfact}(a), 
\begin{align}\label{eq:lambda1r}
\lambda_{1}(G_{1,  a})\ge (1+o(1))\sqrt{(1+ \overline \delta_a)L_p}= (1+o(1)) (1+ \delta_a) \sqrt{L_p}.
\end{align}
Since, $G_{1, 1}, G_{1, 2}, \ldots, G_{1, r}$ are disjoint, the multi-set $\{\lambda_1(G_{1, 1}), \lambda_1(G_{1, 2}), \ldots, \lambda_1(G_{1, r})\}$ is a subset of the spectrum of $G_1$. 
Now, since $G_1= G\backslash H$ (recall that $H$ is the graph obtained by the union of the edges removed from $G$ during the iterative cycle removal process) and $d_{(1)}(H) = O( L_p^{5/12} )$, it follows that 
\begin{align}\label{eq:lambdaG1H}
\lambda_1(H) \leq d_{(1)}(H) = o(\sqrt{L_p}) .
\end{align}
This implies, by Weyl's inequality, $\lambda_a(G)= \lambda_a(G_1) + o(\sqrt{L_p})$, for $1 \leq a \leq r$, which together with \eqref{eq:lambda1r} shows that $G \in \mathrm{UT}_r( (1+o(1)) \overline{\bm \delta})$. 
\end{proof}

 We now move on to proving a matching upper bound on $\mathrm{UT}_r(\bm \delta)$. Again, recall the event $\cB_0$ from Lemma \ref{lem:blem} and denote the event $\cN=\overline \cD_{\log n} \bigcap \overline \cB_0$, where $\cD_{\log n}$ is as defined in  Lemma \ref{lem:dlem}. Combining Lemma \ref{lem:blem} and \ref{lem:dlem} gives, 
\begin{align}\label{eq:N}
\P(\cN)=1-e^{\omega(\log n)}. 
\end{align}
Now, suppose $G \in  \mathrm{UT}(\bm \delta) \bigcap \cN$. As in the proof of the lower bound above, iteratively remove cycles from $G$, in arbitrary order, until $G$ cycle free, and denote by $H$ the graph obtained by the superposition of cycles removed from $G$, and $G_1=G\backslash H$. Note that $d_{(1)}(H) = O(L_p^{5/12} )$, by \eqref{eq:lambdaG1H} and Weyl's inequality, $\lambda_a(G)= \lambda_a(G_1) + o(\sqrt{L_p})$, for $1 \leq a \leq r$ and hence 
\begin{equation}\label{lowerboundspec}
\lambda_a(G_1)\ge (1+\delta_a) (1+o(1)) \sqrt{L_p}.
\end{equation}
Moreover, by definition, $G_1$ is a disjoint union of trees. Denote by $G_{1,1}, G_{1,2}, \ldots, G_{1, \nu}$ the connected components of $G_1$,  such that $|V(G_{1,1})| \geq |V(G_{1,2})| \geq \cdots \geq |V(G_{1, \nu})|$, for some $\nu \geq 1$. Since $G \in \cN$, and in the current regime $L_p= (1+o(1)) \frac{\log n}{\log(1/np)}\ll \log n,$ it follows that $|V(G_{1,b})| \ll \log^2 n$, for all $1 \leq b \leq \nu$, which implies, $\nu \gtrsim n/\log^2 n$. 

Now, similar to the proof of the lower bound where we partitioned the high degree vertices depending on which tree component they came from, we now partition the largest $r$ eigenvalues accordingly. Thus, as before, fix a partition $\Pi$  of the set $[r]=\{1, 2, \ldots, r\}$ with parts $\Pi=\{\pi_1, \pi_2, \ldots, \pi_B\}$, for some $B \geq 1$. 
Given a partition $\Pi$ as above, denote by $\cT_\Pi$ the event that, for all $1 \leq b \leq B$, 
\begin{align}\label{eq:pi}
\lambda_{i_a}(G_1)=\lambda_a(G_{1, \phi(b)}), \text{ for all } 1 \leq a \leq |\pi_b|, 
\end{align} 
where $\phi: \{1, 2, \ldots, B\} \rightarrow  \{1, 2, \ldots, \nu\} $ is some injective function and $\pi_b=\{i_1, i_2, \ldots, i_{|\pi_b|}\}$ with the indices arranged in increasing order, that is, $i_1 < i_2 < \ldots < i_{|\pi_b|}$. In other words, $\cT_\Pi$ prescribes that for each $1 \leq b \leq B$, all the eigenvalues $\lambda_{s}$, with $s \in \pi_b$, belong to the spectrum of a single distinct tree component denoted as $G_{1, \phi(b)}$. Then, for any $1 \leq b \leq B$, and $G \in  \mathrm{UT}(\bm \delta) \bigcap \overline \cB_0$, 
\begin{align*}
2\sum_{a \in \pi_b} (1+\delta_a)^2 (1+o(1)) L_p  \le  2\sum_{a \in \pi_b} \lambda_a^2(G_1) 
& = 2\sum_{a=1}^{|\pi_b|}\lambda_a^2(G_{1, \phi(b)})\\
& \leq  \sum_{a=1}^{|V(G_{1, \phi(b)})|} \lambda_a^2(G_{1, \phi(b)})   \\
& = 2 |E(G_{1, \phi(b)})| = 2 (|V(G_{1, \phi(b)})|-1). 
\end{align*} 
To see this note that the first inequality follows from the lower bound in \eqref{lowerboundspec}. The second equality is by \eqref{eq:pi} and the third inequality uses the fact that the spectrum of $G_{1, \phi(b)}$ is symmetric, since it is a tree and, hence, bipartite \cite[Proposition 4.5.4]{spectral_graph_theory}.  
The fourth equality is a standard fact for any graph which states that the second moment of the spectrum, which is the Frobenius norm square of the adjacency matrix, is equal to twice the number of edges. The final equality is because the number of edges in any tree is one less than the number of vertices. 
Thus, it follows that $$|V(G_{1, \phi(b)})| \geq \left(\sum_{a \in \pi_b} (1+\delta_a)^2  (1+o(1)) L_p\right) + 1.$$ 
(As before, the additional plus $1$ which can be achieved by adjusting the $o(1)$ term is written to help invoke Lemma \ref{lem:dlem} verbatim).  For notational brevity, denote $w_\Pi(b)=(1+o(1)) \sum_{a \in \pi_b} (1+\delta_a)^2$. Therefore, the event $\cT_\Pi\bigcap \left( \mathrm{UT}(\bm \delta) \bigcap \cN  \right)$ implies the event $\cD_{w_\Pi(1), \ldots, w_\Pi(B)}$. This implies, by using Lemma \ref{lem:dlem}, $\sum_{b=1}^B w_\Pi(b) = (1+o(1)) \sum_{a=1}^r (1+\delta_a)^2,$ and $|B|\leq r$, 
\begin{align*}
\P\left(\cG_{n, p} \in \cT_{\Pi} \bigcap \left( \mathrm{UT}(\bm \delta) \bigcap \cN  \right) \right) & \leq \P(\cD_{w_\Pi(1), \ldots, w_\Pi(B)})  \nonumber \\  
& \leq \exp\left\{- \left( \sum_{a=1}^r (1+\delta_a)^2 - r  \right) \log n  + o(\log n) \right\}.   
\end{align*} 
Now, taking a union bound over the set of all partitions
$\cP_r$ of $[r]=\{1, 2, \ldots, r\}$ and using $|\cP_r| \lesssim_r 1$ gives, 
\begin{align}\label{eq:utN}
\P( \cG_{n, p} \in \mathrm{UT}(\bm \delta) \bigcap \cN ) & = \P\left(\bigcup_{\Pi \in \cP_r}\cT_{\Pi} \bigcap \left(\mathrm{UT}(\bm \delta) \bigcap \cN \right)\right)  \nonumber \\ 
& \lesssim_r \exp\left\{- \sum_{a=1}^r \left( \delta_a^2 + 2 \delta_a   \right) \log n  + o(\log n) \right\}.  
\end{align} 
{Finally, using $\P(\cG_{n, p} \in \mathrm{UT}_r(\bm \delta)) \geq e^{-\Theta(\log n)}$ (see \eqref{eq:ut_lb_B}) and \eqref{eq:N} implies, $\P(\cG_{n, p} \in  \cN| \cG_{n, p} \in \mathrm{UT}_r(\bm \delta)) \rightarrow 1$.} Therefore, by \eqref{eq:utN}, 
\begin{align*}
\P(\cG_{n, p} \in \mathrm{UT}_r(\bm \delta))  = (1+o(1)) \P\left(\cG_{n, p} \in \mathrm{UT}_r(\bm \delta) \bigcap  \cN   \right) \leq \exp\left\{- (1+o(1)) \sum_{a=1}^r (\delta_a^2 + 2 \delta_a)  \log n  \right\}, 
\end{align*}
which gives the upper bound. \\

\begin{remark}\label{remark:ut_marginal} {\em The proof of Theorem \ref{thm:1ut} shows that conditional on the upper tail event $\mathrm{UT}(\bm \delta)$, the random graph looks   
(after removing a graph with `small' spectral norm and maximum degree) either like a disjoint union of stars (in the case  $e^{-(\log \log n)^2} \leq n p \ll \sqrt{\log n/\log \log n}$), or a disjoint union of trees of specific sizes (in the case $(\log \log n)^2  \leq \log(1/n p)  \ll \log n$)   This structural information can be used to obtain the marginal large deviations for the upper tail of the edge eigenvalues as well. More generally, from the proof of Theorem \ref{thm:1ut}  one easily can show that, for $1 \leq t \leq r$ and $0=i_0 \leq i_1 < i_2 < \cdots < i_t \leq r$,  
$$\lim_{n \rightarrow \infty}\frac{-\log \P(\lambda_{i_1}(\cG_{n, p}) \geq (1+\delta_{i_1})  \sqrt{L_p}, \ldots, \lambda_{i_t}(\cG_{n, p}) \geq (1+\delta_{i_t} )  \sqrt{L_p} ) }{\log n} = \sum_{s=1}^t (i_{s}-i_{s-1})(2 \delta_{i_s} + \delta_{i_s}^2),$$ 
where $\delta_{i_1} \geq \delta_{i_2} \geq \cdots \geq \delta_{i_t} >0$. This follows by observing that the marginal event $\{\lambda_{i_1}(\cG_{n, p}) \geq (1+\delta_{i_1})  \sqrt{L_p}, \ldots, \lambda_{i_t}(\cG_{n, p}) \geq (1+\delta_{i_t} )  \sqrt{L_p} \}$ is equivalent to the full event  $\mathrm{UT}_t(\bm \delta)$, where 
$$\bm\delta=(\delta_{i_1}, \ldots, \delta_{i_1}, \delta_{i_2}, \ldots, \delta_{i_2}, \ldots, \delta_{i_t}, \ldots, \delta_{i_t}),$$
where the $\delta_{i_s}$ is repeated $i_s-i_{s-1}$ times, for $1 \leq s \leq t$. }
\end{remark}

\subsection{Alternate Approach via Strong Concentration Results}
\label{sec:pf_concentration}
 
Another potential approach to proving the large deviations of the upper tail event $\mathrm{UT}_r(\bm \delta)$ is via the strong concentration results of the spectral norm  of sparse graphs established in \cite{vershynin}. It is well known that the spectral norm of the adjacency matrix $||A(\cG_{n, p})||$ concentrates arounds the spectral norm of its expected value $||\E(A(\cG_{n, p}))||= (1+o(1)) np$ in the dense regime \cite{spectral_techniques}.  On the other hand, as we have already seen, in the sparse regime \eqref{eq:rangep}, one cannot expect such a result, because the spectral norm of $\cG_{n, p}$ is dictated by the maximum degree which is typically much larger than $np$. However, the results of \cite{vershynin} indicate that this is the only barrier and a concentration about $||\E(A(\cG_{n, p}))||$ indeed holds after suitable regularization by pruning the high degree vertices.  More precisely, among other things, they prove the following result (restated in our notation). 

\begin{thm}\label{thmversh}\cite[Theorem 2.1]{vershynin}  Let $\cG_{n, p}^-$ be the subgraph of $\cG_{n, p}$ obtained by removing any subset consisting of at most $10/p$ vertices from $\cG_{n, p}$. Then, for any $b \ge 1$, the following holds with probability at least $1- n^{-b}$, 
$$\|A(\cG_{n, p}^{-}) - \E(A(\cG_{n, p})) \|\le K b^{3/2}\left(\sqrt{np}+\sqrt{d_{(1)}(\cG_{n, p}^{-})} \right),$$
for some absolute constant $K >0$. 
\end{thm}

Although our arguments in Section \ref{sec:ut_pf_I} yield precise structural information about the random graph conditioned on the rare events, and can be easily adapted to address various large deviation questions, including both upper tail and lower tail probabilities, to illustrate the promise of Theorem \ref{thmversh} we sketch below, omitting various technicalities,  an alternate proof strategy for the upper tail large deviation of $\lambda_1(\cG_{n, p})$, in the regime  $p=c/n$. \\ 

\noindent {\it Alternate Proof Sketch of Theorem \ref{thm:1ut} for $r=1$ and $p=c/n$}: To begin with observe that the lower bound on the upper tail probability can be easily derived when $r=1$, since, by Lemma \ref{lem:gfact}(a),  
\begin{align}\label{eq:ut_lb_1}
\P(\lambda_1(\cG_{n, p}) \geq (1+\delta_1) \sqrt{L_p}) \geq \P(d_{(1)}(\cG_{n, p}) \geq (1+\delta_1)^2 L_p)= e^{ -(2\delta_1+\delta_1^2) \log n + o(\log n) } , 
\end{align} 
where the last step uses Proposition \ref{ppn:degree_ut}. 

For the upper bound, suppose $G \sim \cG_{n, p}$, fix $\varepsilon \in (0, \frac{1}{4})$, and denote by $\cV_{\e}$ the set of all vertices in $G$ with degree at least $\e L_p$ (we will refer to these as the `high-degree' vertices). The proof strategy of the upper bound now entails the following three steps: 

\begin{itemize}

\item {\it Bounding the number of high-degree vertices}: Note that when $p=c/n$, then $L_p = (1+o(1)) \frac{\log n}{\log \log n}$, and, hence, $\P(v \in \cV_{\varepsilon})= \P(\dBin(n-1, c/n) \geq \varepsilon L_p)= e^{\varepsilon \log n + o(\log n)}$. Then it is not hard to show that there $\kappa=\kappa(\e)>0$ such that, for all large enough $n$,
\begin{align}\label{eq:degree_Lp}
\P(|\cV_{\e}|\ge n^{1-\kappa})\le e^{-O_{\varepsilon}(n^{1-\kappa})}.
\end{align} 
(Heuristically, $|\cV_{\e}|$ is approximately a Poisson random variable with mean $n^{1-\e + o(1)}$, hence, taking $\kappa < \varepsilon$ small enough should imply \eqref{eq:degree_Lp}.) Denote by $\cE_1$ the event $\{|\cV_{\e}|\le n^{1-\kappa}\}$. Then using \eqref{eq:ut_lb_1} and \eqref{eq:degree_Lp} gives, 
\begin{align}\label{eq:event1}
\P( \cE_1 | \mathrm{UT}_1(\delta_1))=1-o(1). 
\end{align}

\item {\it Pruning edges between low-degree vertices}: The next step is to show that conditional on the upper tail event, the spectral norm of the subgraph of $G$ obtained by removing the edges between $\overline{\cV_{\varepsilon}}$ remains large with high probability. To this end, we denote by $G^{-}$ the graph obtained by removing from $G$ all the vertices in $\cV_{\varepsilon}$ and the edges incident on them. Then the removed graph $G^+=G\backslash G^{-}$ is the subgraph of $G$ consisting of all edges which has at one least endpoint with degree at least $\varepsilon L_p$. Now, since $\E(A)$ is the matrix all of whose entries are exactly $p$, except the diagonal entries, $\|\E(A)\|\le np$. Therefore,  by Theorem \ref{thmversh},  
\begin{align}\label{eq:ut_V}
\P\left(\|A(G^-)\|\ge c +K b^{3/2}(\sqrt{c} +\sqrt{\e L_p})\right) \le e^{ - b \log n}.
\end{align} 
Now, suppose $\cE_2:= \{|| A(G^+)|| \geq (1+\delta_1)\sqrt{L_p}- c - K b^{3/2}(\sqrt{c} + \sqrt{\e L_p } )    \}$ and $G \in  \mathrm{UT}_1(\delta_1) \bigcap \overline{\cE_2}$. Then, by triangle inequality, 
$G \in  \{\|A(G^-)\| \geq c +K b^{3/2}(\sqrt{c} +\sqrt{\e L_p}) \}$. This implies, choosing $b > 2\delta_1+ \delta_1^2 + 1$ and using \eqref{eq:ut_lb_1} and \eqref{eq:ut_V}, 
\begin{align}\label{eq:event2}
\P( \cE_2 | \mathrm{UT}_1(\delta_1))=1-o(1). 
\end{align}

\item {\it High degree vertices have small overlap}:  The final step is to show that the neighborhoods of the vertices in $\cV_\varepsilon$ have small overlaps. For this one can follow the proof of \cite[Lemma 2.5]{bbk} and show the following: For any $b >0,$ there exists a constant $C(\varepsilon)$ such that 
\begin{align}\label{eq:G_neighbor}
\P\left(\exists v\in \cV_{\e}  \text{ such that } N_G(v) \bigcap \left\{\cV_{\e}\bigcup_{u\in \cV_{\e}\setminus \{v\}}N_G(u) \right\}\ge C(\varepsilon) \right) \leq e^{- b \log n}. 
\end{align}
We denote the complement of the event above by $\cE_3$. Then choosing $b$ large enough and using \eqref{eq:ut_lb_1} and \eqref{eq:G_neighbor}, it follows that 
\begin{align}\label{eq:event3}
\P( \cE_3 | \mathrm{UT}_1(\delta_1))=1-o(1).
\end{align}
\end{itemize}
The arguments above can be combined to show that the event $\mathrm{UT}_1(\delta_1) \bigcap \cE_1 \bigcap \cE_2 \bigcap \cE_3 $ implies the event $\mathrm{degUT}_1((1+o(1)) \bar \delta_1)$, where $\bar \delta_1= (1+\delta_1 + \Theta_{r, c}(\sqrt{\varepsilon}) + o(1) )^2-1$. 
The required upper bound can then be derived by using Proposition \ref{ppn:degree_ut} and 
by taking limit as $n \rightarrow \infty$ followed by $\varepsilon \rightarrow 0$.

\section{Proof of Theorem \ref{thm:lt}}
\label{sec:lt_pf}

\subsection{Lower Bound on $\P(\mathrm{LT}_r(\bm \delta))$}
\label{sec:lt_I}
Recall the definitions of the lower tail events $\mathrm{LT}_r(\bm \delta)$ and $\mathrm{degLT}_r(\bm \delta)$ from \eqref{eq:lt_delta} and \eqref{eq:degree_lt_delta}, respectively.  The following lemma gives a lower bound of the lower tail probability in terms of the probability of degree lower tail event. 

\begin{lem}\label{lm:lt_lb} For $0 < \delta_1 \leq \delta_2 \leq \cdots \leq \delta_r < 1$, 
\begin{align}
\P(\mathrm{LT}_r(\bm \delta)) & \geq (1+o(1))  \P( \mathrm{degLT}_r((1+o(1)) \overline{\bm \delta})), \nonumber 
\end{align}
where $\overline{\bm \delta} = 1-(\bm 1 - \bm \delta)^2 $
\end{lem}

The lower bound on $\P(\mathrm{LT}_r(\bm \delta))$ in Theorem \ref{thm:lt} now follows from Proposition \ref{ppn:deglt}.

\subsubsection{\textbf{Proof of Lemma \ref{lm:lt_lb}}}
\label{sec:lt_pflb}
We rely heavily on the FKG inequality. To this end, define 
\begin{align}\label{eq:event_D}
\cD:= \left\{F \in \sG_n: F= F_1 \bigcup F_2 \text{ such that } 
~\begin{array}{l}
 E(F_1)\bigcap E(F_2) = \emptyset  \\
 F_1 \text{ is a disjoint union of stars }   \\
  ||F_2|| = o(\sqrt{L_p}) \text{ and } d_{(1)}(F_2)=o(L_p) 
\end{array} 
\right\}.
\end{align}
Thus in words, $\cD$ is the event that the graph admits a decomposition into two disjoint graphs where the former is a disjoint union of stars and the  latter has negligible maximum degree as well as spectral norm. We start with the following easy but useful fact.

\begin{lem}\label{lm:events_II} The events $\cD$ and $\mathrm{degLT}_r(\bar{\bm \delta})$ are decreasing events. 
\end{lem}

\begin{proof} First we show $\cD$ is decreasing. Take $G_1 \in \cD$ and $G_2$ a subgraph of $G_1$. By definition of $\cD$, $G_1= F_1 \bigcup F_2$ such that $F_1$ is a disjoint union of stars, and $F_2$ is such that $||F_2|| = o(\sqrt{L_p})$ and $d_{(1)}(F_2)=o(L_p)$. Take $F_3$ and $F_4$ to be the corresponding subgraphs of $F_1$ and $F_2$ in $G_2$. Observe that, $||F_4|| \leq ||F_2||=o(\sqrt{L_p})$ by Lemma \ref{lem:gfact} and $d_{(1)}(F_4) \leq d_{(1)}(F_2)=o(L_p)$ as the maximum degree can only decrease in a subgraph. Finally, since any subgraph of a disjoint union of stars is also a disjoint union of stars, we are through.

To show $\cE := \mathrm{degLT}_r(\bar{\bm \delta})$ is a decreasing  event, take $F_1 \in \cE$ and $F_2 \in \sG_n$ with  $E(F_2) \subset E(F_1)$. Label the vertices of $F_1$ by $v_1, \ldots, v_n$ in non-increasing order of their degrees in $F_1$. Similarly, suppose $u_1, \ldots, u_n$ are the vertices of $F_2$ in non-increasing order of their degrees in $F_2$. Now, fix a vertex $u_i$ in $F_2$ and assume that it corresponds to the vertex $v_j$ in $F_1$. If $i=j$, then $d_{F_2}(u_i) \leq d_{F_1}(v_i)$, since $F_2$ is a subgraph of $F_1$. If $i <j$, then $d_{F_2}(u_i) \leq d_{F_1}(v_j) \leq d_{F_1}(v_j)$. Finally, $i >j$, then there exists $i' < i \leq j'$ such that the vertex $u_{i'}$ in $F_2$ corresponds to the vertex $v_{j'}$ in $F_1$. Then we have, $d_{F_2}(u_i) \leq d_{F_2}(u_{i'}) \leq d_{F_1}(v_{j'}) \leq d_{F_1}(v_i)$. This shows that $i$-th largest degree in $F_2$ is at most the $i$-th largest degree in $F_1$, showing that $\cE$ is a decreasing event.  
\end{proof}

Next, we show that the events  $\cD$ and $\mathrm{degLT}_r(\bar{\bm \delta})$ together imply the lower tail event $\mathrm{LT}_r(\bm \delta  )$, up to a $(1+o(1))$ factor.  

\begin{lem} The event $\cD \bigcap \mathrm{degLT}_r((1+o(1)) \bar{\bm \delta})$ implies the eigenvalue lower tail event $\mathrm{LT}_r(\bm \delta )$. 
\end{lem}

\begin{proof} Take $0 < \delta_1' \leq \delta_2' \leq  \ldots \leq \delta_r' < 1$ and $F \in \cD \bigcap \mathrm{degLT}_r(\overline{\bm \delta'})$, where $\overline{\bm \delta'} = \bm 1-(\bm 1 - \bm \delta')^2 $. Then  $F = F_1 \bigcup F_2$ as in \eqref{eq:event_D} and, by Weyl's inequality,  
$$\lambda_r(F) \leq \sqrt{d_{(r)}(F_1)}+ o(\sqrt{L_p}) \leq \sqrt{d_{(r)}(F)} + o(\sqrt{L_p}) \leq  (1- \delta_r' + o(1))  \sqrt{L_p},$$ 
which implies $F \in \mathrm{LT}_r(\bm \delta )$, by choosing $\bm \delta'= \bm \delta + o(1)$ in such a way that the $o(1)$-term above vanishes.
\end{proof}

The proof of Lemma \ref{lm:lt_lb} can be completed using the above lemmas. We consider, as in the proof of Theorem \ref{thm:1ut}, the two cases: (1) where  $e^{-(\log \log n)^2} \leq n p \ll \sqrt{\log n/\log \log n}$ and (2) for  $(\log \log n)^2  \leq \log(1/n p)  \ll \log n$. \\

\noindent{{\it Case}} $1$: $e^{-(\log \log n)^2} \leq n p \ll \sqrt{\log n/\log \log n}$. Recall the event $\cM$ from \eqref{eq:utgood} and that $\P(\cM)\ge 1-e^{-\omega(\log n)}$and note that $\cM$
implies the event $\cD$, thanks to Lemma \ref{lem:decompose}. Therefore, 
$$\P(\cD) \geq \P(\cM) \rightarrow 1.$$ Then, by Lemma \ref{lm:events_II} and the FKG inequality, $\P(\cD|\mathrm{degLT}_r((1+o(1)) \bar{\bm \delta})) \geq \P(\cD) \rightarrow 1$, which implies, 
\begin{align*}
\P(\mathrm{LT}_r(\bm \delta  ) ) \geq \P\left(\cD \bigcap \mathrm{degLT}_r((1+o(1))  \bar{\bm \delta})\right) =(1+o(1)) \P(\mathrm{degLT}_r((1+o(1))  \bar{\bm \delta})), 
\end{align*}
completing the proof of Lemma \ref{lm:lt_lb}, when $e^{-(\log \log n)^2} \leq n p \ll \sqrt{\log n/\log \log n}$. \\

\noindent{{\it Case $2$}} : $(\log \log n)^2  \leq \log(1/n p)  \ll \log n$. In this case, even though we cannot invoke Lemmas \ref{lem:largecycle} and \ref{lem:clem}, we will modify the proof of Lemma \ref{lem:decompose} to directly argue that $\P(\cD) \rightarrow 1$. To this end, recall the decomposition of the vertices of $G$ in to three components $X_1, X_2, Y$, as in the proof of Lemma \ref{lem:decompose} as illustrated in Fig \ref{fig:graph_decomposition}. 

\begin{lem}\label{lem:lt_psmall} For  $(\log \log n)^2  \leq \log(1/n p)  \ll \log n$, $\P(\cD) \rightarrow 1$.
\end{lem}

\begin{proof} Let  $G \sim  \cG_{n, p}$, be a realization of the random graph, and $\cT$ be the event that $G$ is a forest and each component of $G$ is of size at most $(1+o(1)) \Delta_p$. It follows from \cite[Lemma 2.2]{KS} that $\P(\cT) \rightarrow 1$. Next, recall the event $\cA_2$ from \eqref{eq:A2} and the event $\cC$ from Lemma \ref{lem:clem}. To begin with assume $G \in \cT \bigcap \left( \cA_2   \bigcup \cC \right)$. There are two cases: 

\begin{itemize}

\item $G \in  \cT \bigcap \cA_2 $: Then there exists a vertex  $v \in X_1$ which has at least $\Delta_p^{7/8}$ neighbors in $X_1 \bigcup X_2$ (recall the definitions of $X_1$ and $X_2$ from the proof of Lemma \ref{lem:decompose}). Note that every vertex in $X_1 \bigcup  X_2$ has degree at least $\Delta_p^{1/3}$, which implies that the induced subgraph of $G$ on the set vertices\footnote{For $S \subseteq [n]$, $N_G(S)$ denotes the neighborhood of $S$ in $G$, that is, $N_G(S)=\{v \in [n]: (u, v) \in E(G) \text{ for some } u \in S\}. $} 
$$v \bigcup N_{G}(v) \bigcup N_G(N_{G}(v)) $$ is a tree of size at least $\Delta_p^{7/8} (\Delta_p^{1/3}- 1 ) \gg \Delta_p$, which contradicts $G \in \cT$. This shows that $\cT \bigcap \cA_2$ is empty.

\item $G \in  \cT \bigcap \cC $: Then there exists a vertex  $v \in [n]$ which has at least $\Delta_p^{1/3}$ other vertices of $G$, each with degree greater than $\Delta_p^{3/4}$, within distance at most two. This implies the induced subgraph of $G$ on the set vertices 
$$v \bigcup N_{G}(v) \bigcup N_G(N_{G}(v)) \bigcup N_G(N_G(N_{G}(v)))$$ is a tree of size at least $\frac{1}{2}\Delta_p^{1/3} (\Delta_p^{3/4}- 1) \gg \Delta_p$, which contradicts $G \in \cT$. To see the last deduction, note that by hypothesis either $N_{G}(v)$ or $N_G(N_{G}(v))$ contains $\frac{1}{2}\Delta_p^{1/3}$ vertices of degree at least $\Delta_p^{3/4}$. This shows that $\cT \bigcap \cC$ is empty.  

\end{itemize} 
The two cases above show that $\cT \bigcap \left( \cA_2   \bigcup \cC \right)$ is empty. Therefore, it suffices to assume that  $G \in  \cT \bigcap  \overline \cA_2   \bigcap \overline \cC$.  We will show below that $\cT \bigcap  \overline \cA_2   \bigcap \overline \cC  \subseteq \cD$. This will complete the proof of the lemma, since 
$$\P(\cD) \geq \P \left( \cT \bigcap  \overline \cA_2   \bigcap \overline \cC \right) = \P(\cT) \rightarrow 1.$$

To show $\cT \bigcap  \overline \cA_2   \bigcap \overline \cC \subseteq \cD$, assume $G \in \cT \bigcap  \overline \cA_2   \bigcap \overline \cC$.  As in the proof of Lemma \ref{lem:decompose}, decompose the vertices of $G$ in to three components $X_1, X_2, Y$, and consider the graphs $G_1, G_2, G_3, G_4$ as in Fig \ref{fig:graph_decomposition}. Also, as in the proof of Lemma \ref{lem:decompose}, partition the vertices of $Y$ into two parts $Y=T_1\bigcup T_2$, where $T_1$ is the set of vertices in $Y$ with degree 1 in the graph $G_3$, and $T_2= Y \backslash T_1$. Since, $G \in \cT \bigcap  \overline \cA_2   \bigcap \overline \cC,$ we can assume that for every $v \in X_1$, the number edges from $v$ to $T_2$ in the graph  $G_3$ is at most $\Delta_p^{1/3}$, since otherwise the fact that $G\in \cT$ would imply $G\in \cC.$

  Then we split the graph $G_3$ in to two parts $F_1$ and $G_{3, 2}$, where $F_1$ is the set of all edges from $X_1$ to $T_1$ and $G_{3, 2}$ the set of edges from $X_1$ to $T_2$. Note that the maximum degree of the graph $G_{3, 2}$ is $np(1+1/\log\log n)+\Delta^{1/3}_p=o(\sqrt{L_p})$ as argued in Lemma \ref{lem:decompose}. This is because the maximum degree of any vertex in $Y$ is at most this and so is the maximum degree of any vertex in $X_1$ in to $T_2$ by the above discussion. Hence $\lambda_1(G_{3, 2}) \leq d_{(1)}(G_{3, 2})=o(\sqrt{L_p})$. Moreover,  
 since $G_1, G_2$ and $G_4$ are all forests and $\max\{d_{(1)}(G_1), d_{(1)}(G_2), d_{(1)}(G_4)\}=o(L_p)$, for $G \in \cT \bigcap  \overline \cA_2   \bigcap \overline \cC$,  by Lemma \ref{lem:gfact}(c), 
$$||G_1|| \leq 2 \sqrt{d_{(1)}(G_1)} =o(\sqrt{L_p}), \quad ||G_2|| =o(\sqrt{L_p}), \quad \text{and} \quad  ||G_4|| =o(\sqrt{L_p}).$$ 
Here the bounds on $\max\{d_{(1)}(G_2), d_{(1)}(G_4)\}$ follow from definition while the bound on $d_{(1)}(G_1)$ follows since we are on $\overline \cA_2. $ Therefore, setting $F_2=G_1\bigcup G_2 \bigcup G_4\bigcup G_{3, 2}$, it follows that $||F_2||=o(\sqrt{L_p})$. Finally, noting that $F_1$ is a disjoint union of stars, it follows that $G \in \cD$, as required. 
\end{proof}

The proof of Lemma \ref{lm:lt_lb} can now be completed as before: By Lemma \ref{lm:events_II} and the FKG inequality, $\P(\cD|\mathrm{degLT}_r((1+o(1)) \bar{\bm \delta} )) \geq \P(\cD) \rightarrow 1$, where the last step uses Lemma \ref{lem:lt_psmall} above. This shows, $$\P(\mathrm{LT}_r(\bm \delta  ) ) \geq (1+o(1)) \P(\mathrm{degLT}_r((1+o(1)) \bar{\bm \delta})),$$ for $(\log \log n)^2  \leq \log(1/n p)  \ll \log n$, completing the proof of Lemma \ref{lm:lt_lb}. \hfill $\Box$ \\

\subsection{Upper bound on $\P(\mathrm{LT}_r(\bm \delta))$}
\label{sec:lt_II}

To begin with, observe that if $r=1$, then the upper bound follows trivially from 
Lemma \ref{lem:gfact}(a) and Proposition \ref{ppn:deglt},
$$\P(\lambda_1(\cG_{n, p}) \leq (1-\delta_1) \sqrt{L_p}) \leq \P(d_{(1)}(\cG_{n, p})\leq (1-\delta_1)^2 L_p) = e^{- n^{2\delta_1-\delta_1^2 + o(1)}},$$
which gives the desired limit. 
For general $r$, however, the situation is more delicate. Here, the key step is the following Ramsey-type result which shows, fixing $\varepsilon > 0$ small enough depending on $r$, if for some $K$ (depending on $\varepsilon$), the $K$-th largest degree of $\cG_{n, p}$ is $\varepsilon$ more larger than a certain atypical value, then either the largest eigenvalue $\lambda_1(\cG_{n, p})$ or the $r$-largest eigenvalue $\lambda_r(\cG_{n, p})$ will be large enough to violate the lower tail event.

\begin{lem}\label{lm:r_I} 
Given $r \geq 1$ and $\bm \delta=(\delta_1, \delta_2, \ldots, \delta_r)$ as above, choose $0 < \varepsilon <  \min \{ (1-\delta_r)^2/8 r, 1/r^{1/4}, \frac{1}{4} \}$ and $\e_1 > 0$ such that $ \e_1 \le \min\{ \e^2/r,  \frac{1}{4} \}$. Then for $K \geq (4/\varepsilon_1^4+1) 4^{1/\varepsilon^4}$ and for $n$ large enough, the event $\{G \in \sG_n : d_{(K)}(G) \geq (1-\delta_r + \varepsilon)^2 L_p \} \bigcap \{G \in \sG_n : d_{(1)}(G)\leq L_p\}$, implies  the event 
$$\left\{\{G \in \sG_n: \lambda_1(G) \geq 2 \sqrt{L_p}\} \bigcup \{ G \in \sG_n: \lambda_r (G) \geq (1-\delta_r + \e/2 ) \sqrt{L_p}\} \right\}.$$ 
\end{lem}

The proof of the lemma is given below in Section \ref{sec:pfr_I}.  First, we show  how this can be used to complete the proof of upper bound of the lower tail probability. To this end, note that, since $\lambda_1(\cG_{n, p}) \geq \sqrt{d_{(1)}(\cG_{n, p})}$,  the event $\{d_{(1)}(\cG_{n, p}) > L_p\}$ implies the event $\{\lambda_1(\cG_{n, p}) > \sqrt{L_p}\}$. This shows, $\P(\mathrm{LT}_r(\bm \delta) \bigcap \{d_{(1)}(\cG_{n, p}) > L_p\}) =0$. Then, by Proposition \ref{ppn:deglt}, 
\begin{align*}
\P( \mathrm{LT}_r(\bm \delta)  )
& \leq \P(\{d_{(K)}(\cG_{n, p}) \leq (1-\delta_r + \varepsilon)^2 L_p\}) +  \P\left(\mathrm{LT}_r(\bm \delta) \bigcap \{d_{(K)}(\cG_{n, p}) \geq (1-\delta_r + \varepsilon)^2 L_p\}\right) \nonumber \\ 
& \leq e^{-n^{(1+o(1)) (2\delta_r - \delta_r^2 + O(\varepsilon))}} +  \P\left(\mathrm{LT}_r(\bm \delta) \bigcap \{d_{(K)}(\cG_{n, p}) \geq (1-\delta_r + \varepsilon)^2 L_p\} \bigcap \{d_{(1)}(\cG_{n, p})\leq L_p\}\right) \nonumber \\ 
&=e^{-n^{(1+o(1)) (2\delta_r - \delta_r^2 + O(\varepsilon))}},   
\end{align*}
where the last step  uses
$$\P\left(\mathrm{LT}_r(\bm \delta) \bigcap \{d_{(K)}(\cG_{n, p}) \geq (1-\delta_r + \varepsilon)^2 L_p\} \bigcap \{d_{(1)}(\cG_{n, p})\leq L_p\}\right)=0,$$
by Lemma \ref{lm:r_I}. This completes the proof of the upper bound for the lower tail probability by choosing $\e$ arbitrary small.

\subsection{Proof of Lemma \ref{lm:r_I}}  
\label{sec:pfr_I} 

Let $G \in \sG_n$ and arrange the vertices of $G$ in the decreasing order of the degrees,  labeled as $v_1,  v_2, \ldots, v_n$. Recall that the neighborhood of the vertex $v_i$ in the graph $G$ is denoted as $N_G(v_i)$. 
\begin{lem}\label{lm:overlap_I} There exists a constant $C > 0$ such that for all $0< \varepsilon_1<  \frac{1}{4}$ the following holds. Let $K$ be such that there exists a vertex $v \in V_K:=\{v_1, v_2, \ldots, v_K\}$, with
$$|\{u \in V_K: |N_{G}(v) \cap N_{G}(u)|  > \varepsilon_1 L_p\}| \geq \tfrac{4}{\varepsilon_1^4}.$$
Then on the event $\{G \in \sG_n : d_{(1)}(G) \leq L_p \}$, for $n$ large enough and $\varepsilon_1$ small enough, $\lambda_1(G) \geq 2 \sqrt{L_p}$. 
\end{lem}

\begin{proof}The proof proceeds by constructing a test vector $\phi$ and lower bounding the Rayleigh quotient $\langle \phi, A(G) \phi \rangle/ \|\phi\|^2_2,$  where $\langle x,y \rangle$ denotes the inner product of vectors $x$ and $y$. 

Denote the set of vertices $\widetilde{V}_K:=\{u \in V_K: |N_{G}(v) \cap N_{G}(u)|  > \varepsilon_1 L_p\}$. Then define a function $\phi: V(G) \rightarrow \R$ as follows: 
$$
\phi(x)=
\left\{
\begin{array}{cc}
\sqrt{L_p}  & \text{ if }  x=v  \\
\varepsilon_1^2 \sqrt{L_p}  &   \text{ if }  x \in \widetilde{V}_K\backslash\{v\}    \\
1  &   \text{ if }  x \in N_{G}(v)\backslash \widetilde{V}_K  \\
0 & \text{ otherwise}.
\end{array}
\right.
$$ 
Note that 
\begin{align}\label{eq:phi}
||\phi||_2^2 =\sum_{v \in V(G)} \phi(v)^2 \leq L_p + |\widetilde{V}_K| \varepsilon_1^4 L_p + d_{G}(v) \leq L_p (2+ |\widetilde{V}_K| \varepsilon_1^4),
\end{align} 
where the last step uses the fact $d_{G}(v) \leq L_p$ on the given event. On the other hand,
\begin{align}\label{eq:overlap_I}
\langle \phi, A(G) \phi \rangle  \geq  2 \sum_{x \in \widetilde{V}_K\backslash\{v\}}  \sum_{y \in N_{G}(v)\backslash \widetilde{V}_K} a_{xy} \phi(x) \phi(y) & =   2 \varepsilon_1^2 \sqrt{L_p} \sum_{x \in \widetilde{V}_K\backslash\{v\}}  \sum_{y \in N_{G}(v)\backslash \widetilde{V}_K} a_{xy} \nonumber \\ 
& \geq  2(|\widetilde{V}_K|-1) \varepsilon_1^2 \sqrt{L_p} (\varepsilon_1 L_p- K) \nonumber \\ 
& = 2(|\widetilde{V}_K|-1) \varepsilon_1^3 L_p^{3/2} - O_{K, \varepsilon_1}(\sqrt{L_p}) \nonumber \\ 
& \geq |\widetilde{V}_K|  \varepsilon_1^3 L_p^{3/2},
\end{align}
for any small $\varepsilon_1$ and $n$ large enough depending on $\e_1.$ Then using \eqref{eq:phi} and \eqref{eq:overlap_I} gives, 
\begin{align}
\lambda_1(G) \geq \frac{\langle \phi, A(G) \phi \rangle}{ ||\phi||_2^2} \geq  \frac{ |\widetilde{V}_K|  \varepsilon_1^3 L_p^{3/2} }{L_p (2+ |\widetilde{V}_K| \varepsilon_1^4)} \geq \frac{ |\widetilde{V}_K|  \varepsilon_1^3 \sqrt{L_p}}{2 |\widetilde{V}_K| \varepsilon_1^4} \geq 2 \sqrt{L_p}, \nonumber   
\end{align}
where the third inequality uses $|\widetilde{V}_K| \geq 2/\varepsilon_1^4$ and the last inequality holds since $\varepsilon_1 \le \frac{1}{4}$ in the hypothesis of the lemma. 
\end{proof}

For the remainder of the proof, we will assume that the events $d_{(K)}(G) \geq (1-\delta_r + \varepsilon)^2 L_p$ and $d_{(1)}(G)\leq L_p$ hold.  The next lemma says if $K$ is large enough, either $\lambda_1(G)$ is large or the neighborhoods of many high degree vertices have small mutual overlap. 
\begin{lem} Fix any $L \geq 1$ and $\varepsilon_1 >0$, with $K \geq (4/\varepsilon_1^4+1)L$. Then, either  $\lambda_1(G) \geq 2 \sqrt{L_p}$ or there exists a set of vertices $ V'_L:=\{v_{j_1}, v_{j_2}, \ldots, v_{j_L} \} \subseteq V_K$, such that $$N_{G}(v_{j_a}) \cap N_{G}(v_{j_b}) \leq \varepsilon_1 L_p,\quad \text { for } 1 \leq a < b \leq L.$$
\end{lem}

\begin{proof} 
The argument proceeds by first constructing  an auxiliary graph $H=(V(H), E(H))$, with $V(H)=\{1, 2, \ldots, K\}$ and an edge between $(i, j)$, for $1 \leq i < j \leq K$, if $$|N_{G}(v_i) \cap N_{G}(v_j)| \geq \varepsilon_1 L_p.$$ 
Thus the lemma will be proved once we exhibit that that $H$ admits an independent set of size at least $L.$
Now, it suffices to assume that the maximum degree of the graph $H$ is $4/\varepsilon_1^4$ (since, otherwise, by Lemma \ref{lm:overlap_I},   $\lambda_1(G) \geq 2 \sqrt{L_p}$). Hence, the graph $H$ has an independent of size at least $\frac{K}{4/\varepsilon_1^4+1} \geq L$, for $\varepsilon_1$ small enough and we are done. 
\end{proof}

Now, choose $\varepsilon_2 = \varepsilon^2$ where $\e$ appears in the assumption on $d_K(G),$ and construct the graph $F=(V(F), E(F))$ as follows: $V(F)=\{j_1, j_2, \ldots, j_L\}$ and there is an edge between $(j_a, j_b)$ if\footnote{Given a graph $G=(V(G), E(G))$ and disjoint subsets $S, T \subset V(G)$, $E(S, T)$ denotes the set of edges with one endpoint $S$ and the other end point in $T$.}
\begin{align}\label{eq:edge_ab}
\left|E\left(N_{G}(v_{j_a})\backslash \left(N_{G}(v_{j_a}) \bigcap N_{G}(v_{j_b}) \right), N_{G}(v_{j_b}) \backslash \left(N_{G}(v_{j_a}) \bigcap N_{G}(v_{j_b}) \right)\right) \right | \leq  \varepsilon_2 L_p^{3/2}.
\end{align}
(In other words, there is an edge between  $(j_a, j_b)$ if the number of edges between the non-overlapping neighborhoods of $v_{j_a}$ and $v_{j_b}$ is at most $\varepsilon_2 L_p^{3/2}$.) Choose $L = 4^{r'}$, where $r'=1/\varepsilon_2^2$.
Then by Ramsey's theorem \cite{conlon_fox_sudakov}, one of the following holds: (1) there is an independent set in $F$ of size at least $r'$, or (2) there is a clique in $F$ of size at least $r'$. By choosing a careful dependence across the different parameters we will show Case 1 implies $\lambda_1(G)\ge 2\sqrt{L_p}$ while Case 2 implies $\lambda_2(G)\ge (1-\delta_r+ \varepsilon/2)\sqrt{L_p}.$ The arguments relies on the variational characterization of the spectrum and bounding the appropriate Rayleigh quotients. \\

\noindent
{\it Case } 1: In this case $F$ has an independent set of size $r'$.  Without loss of generality assume that the vertices $\sX:=\{v_{j_1}, v_{j_2}, \ldots, v_{j_{r'}}\}$ form an independent set.  Now, let $H$ be the subgraph of $G$ obtained by deleting all edges among the vertices $V'_L$ and all edges with one endpoint in $V_K$ and the other endpoint in  
$$\bigcup_{1\leq a \ne b \leq L}N_{G}(v_{j_a}) \bigcap N_{G}(v_{j_b}).$$ Thus by construction in $H,$ the set of vertices $V'_L$ is an independent set, and the neighborhoods $\{N_{H}(u): u \in V'_L\}$ are mutually disjoint. Then define a function $\phi: V(G) \rightarrow \R$ as follows: 
$$
\phi(x)=
\left\{
\begin{array}{cc}
\sqrt{L_p}  & \text{ if }  x \in \sX  \\
1  &   \text{ if }  x \in \bigcup_{u \in \sX} N_{H}(u)  \\
0 & \text{ otherwise}.
\end{array}
\right.
$$ 
Note that on the event $\{d_{(1)}(G) \leq L_p\}$, 
\begin{align}\label{eq:phi_II}
||\phi||_2^2 =\sum_{v \in V(G)} \phi(v)^2 \leq L_p r' + \sum_{v\in \sX} d_{G}(v) \leq 2 r' L_p. 
\end{align} 
Also, for $1 \leq a < b \leq r'$, there are at least $\varepsilon_2 L_p^{3/2} $ edges between $N_{H}(v_{j_a})$ and $N_{H}(v_{j_b})$. This implies, 
\begin{align}\label{eq:overlap_II}
\langle \phi, A(H) \phi \rangle   \geq  2 {r' \choose 2} \varepsilon_2 L_p^{3/2},
\end{align}
Using \eqref{eq:phi_II} and \eqref{eq:overlap_II} gives, 
\begin{align*}
\lambda_1(G) \geq \lambda_1(H)  \geq \frac{\langle \phi, A(H) \phi \rangle}{ ||\phi||_2^2}  \ge  \frac{r'}{4}\varepsilon_2 \sqrt{L_p} \gtrsim 2 \sqrt{L_p},    
\end{align*}
since $\frac{r'}{4}\varepsilon_2 =1/4\varepsilon_2 = 1/4\varepsilon^2 > 2$, since $\e \leq \frac{1}{4}$. \\

\noindent
{\it Case } 2:  In this case $F$ has a clique of size $r'$ and hence of any size $s\le r'$. Note that our choice of $\e$ and that $\e_2=\e^2$ implies that $r'=1/\e_2^2 \ge r.$ 
 Let $\sX:=\{v_{j_1}, v_{j_2}, \ldots, v_{j_{r}}\}$ be $r$ vertices which form a clique. Here, we will construct $r$-functions, which are linearly independent,  $\phi_1, \phi_2, \ldots, \phi_r: V(G) \rightarrow \R$, such that for $\phi=\sum_{a=1}^{r} \gamma_a \phi_a$, we will have $\frac{\langle \phi, A\phi \rangle}{ ||\phi||^2} \geq (1-\delta_r + \varepsilon/2) \sqrt{L_p}$, for all $\gamma_1, \gamma_2, \ldots, \gamma_{r} \in \R$. This would imply by standard variational characterization, that $\lambda_r(G) \ge (1-\delta_r + \varepsilon/2) \sqrt{L_p}$. 
For $1 \leq a \leq r$, define the function $\phi_a: V(G) \rightarrow \R$ as follows: 
\begin{align*}
\phi_a(x)=
\left\{
\begin{array}{cc}
\sqrt{d_{G}(v_{j_a})}  & \text{ if }  x = v_{j_a}  \\
1  &   \text{ if }  x \in N_{G}(v_{j_a}) \backslash \left( \bigcup_{u \in \sX\backslash \{v_{j_a}\}} N_{G}(u)  \right)  \\
0 & \text{ otherwise}.
\end{array}
\right.
\end{align*}
Note that the functions $\phi_1, \phi_2, \ldots, \phi_r$ have disjoint supports and hence are linearly independent. Furthermore, 
\begin{align}\label{eq:norm_phi} 
||\phi||_2^2 = \sum_{a=1}^r \gamma_a^2 ||\phi_a||_2^2 \leq  2 \sum_{a=1}^r  d_{G}(v_{j_a}) \gamma_a^2. 
\end{align}
However, for our purposes we also need the easy lower bound $||\phi||_2^2\ge (1-\delta_r + \varepsilon)^2 L_p \sum_{a=1}^r \gamma_a^2 $.  
where we use $d_{G}(v_{j_a}) \geq d_{(K)}(G) \geq  (1-\delta_r + \varepsilon)^2 L_p$.
Also observe that, for $1 \leq a \leq r$,   
\begin{align}\label{eq:degree_L}
\left|\bigcup_{u \in \sX \backslash \{v_{j_a}\}} \left(N_{G}(u) \bigcap N_{G}(v_{j_a}) \right) \right| \leq r  \varepsilon_1 L_p,
\end{align} 
since $N_{G}(v_{j_a}) \cap N_{G}(u) \leq \varepsilon_1 L_p$, for $ v_{j_a}, u \in V_L$. Using this we get,
\begin{align}
& \langle \phi, A(G) \phi \rangle  \nonumber \\ 
& = \sum_{a=1}^r \gamma_a^2 \langle \phi_a, A(G) \phi_a \rangle +  \sum_{1 \leq a\ne b \leq r} \gamma_a \gamma_b \langle \phi_a, A(G) \phi_b \rangle \nonumber \\ 
& \geq 2 \sum_{a=1}^r \gamma_a^2 \sqrt{d_{G}(v_{j_a})} (d_{G}(v_{j_a}) - r \varepsilon_1 L_p )+  \sum_{1 \leq a\ne b \leq r} \gamma_a \gamma_b \langle \phi_a, A(G) \phi_b \rangle \tag*{(by \eqref{eq:degree_L})} \nonumber \\ 
 & \geq 2 (1-\delta_r + \varepsilon) \sqrt{L_p} \sum_{a=1}^r \gamma_a^2 (d_{G}(v_{j_a}) - r \varepsilon_1 L_p)+  \sum_{1 \leq a\ne b \leq r} \gamma_a \gamma_b \langle \phi_a, A(G) \phi_b \rangle \nonumber \\
 \label{eq:phi_I}
 & \geq 2 (1-\delta_r + \varepsilon) \sqrt{L_p} \left(1- \frac{\e^2}{(1-\delta_r)^2}  \right)\sum_{a=1}^r \gamma_a^2 (d_{G}(v_{j_a}) )+  \sum_{1 \leq a\ne b \leq r} \gamma_a \gamma_b \langle \phi_a, A(G) \phi_b \rangle, 
\end{align}
where the last step uses $d_{G}(v_{j_a}) \geq d_{(K)}(G) \geq  (1-\delta_r + \varepsilon)^2 L_p \geq   (1-\delta_r )^2 L_p$,  and $r\e_1\le {\e^2}$.  Now, recalling \eqref{eq:edge_ab},
note that 
\begin{align*}
\left|\sum_{1 \leq a\ne b \leq r} \gamma_a \gamma_b \langle \phi_a, A(G) \phi_b \rangle \right| & \leq \sum_{1 \leq a\ne b \leq r} |\gamma_a| |\gamma_b| \langle \phi_a, A(G) \phi_b \rangle \nonumber \\
&  \leq  \left(\varepsilon_2  L_p^{3/2}  + L_p \right) \sum_{1 \leq a < b \leq r} 2 |\gamma_a| |\gamma_b|  \leq  r \left(\varepsilon_2  L_p^{3/2}  + L_p \right) \sum_{a=1}^r \gamma_a^2, \nonumber \end{align*}
{where the third inequality uses $2\gamma_a\gamma_b \leq \gamma_a^2 + \gamma_b^2$.}
The extra $L_p$ additive term above comes from the possible edge between $v_{j_a}$ and $v_{j_b}$ contributing $\sqrt{d_{G}(v_{j_a})d_{G}(v_{j_b})}\le L_p$. This implies, 
\begin{align}\label{eq:phi_Lp}
\frac{1}{{||\phi||_2^2} } \left|\sum_{1 \leq a\ne b \leq r} \gamma_a \gamma_b \langle \phi_a, A(G) \phi_b \rangle \right| & \leq \frac{1}{{||\phi||_2^2} } \sum_{1 \leq a\ne b \leq r} |\gamma_a| |\gamma_b| \langle \phi_a, A(G) \phi_b \rangle \nonumber \\ 
&\leq  \frac{r \left(\varepsilon_2  L_p^{3/2}  + L_p \right) \sum_{a=1}^r \gamma_a^2 }{ (1-\delta_r + \varepsilon)^2 L_p \sum_{a=1}^r \gamma_a^2 }  \nonumber \\  
&  \le \frac{r \left(\varepsilon_2  \sqrt{L_p}  + 1 \right)}{(1-\delta_r)^2} \nonumber \\ 
&=\frac{r \left(\e^2 \sqrt{L_p}  + 1 \right)}{(1-\delta_r)^2}  ,
\end{align}
where we use the previously stated lower bound on $\|\phi\|_2^2$ and that $\varepsilon_2 = \varepsilon^2.$  
{We now lower bound the contribution of the diagonal term in \eqref{eq:phi_I} using the upper bound on $\|\phi\|_2^2$ in \eqref{eq:norm_phi}: 
\begin{align*}
& \frac{2 (1-\delta_r + \varepsilon) \sqrt{L_p} \left(1- \frac{\e^2}{(1-\delta_r)^2}  \right)\sum_{a=1}^r \gamma_a^2 (d_{G}(v_{j_a}) )}{||\phi||_2^2} \nonumber \\ 
& \;\;\;\;\;\;\;\;\;\;\;\;\;\;\;\;\;\;\;\;\;\;\;\;\;\;\;\;\;\;\;\;\;\;\;\;\;\;\;\;\;\; \ge \frac{2 (1-\delta_r + \varepsilon) \sqrt{L_p} \left(1- \frac{\e^2}{(1-\delta_r)^2}  \right)\sum_{a=1}^r d_{G}(v_{j_a})  \gamma_a^2 }{2 \sum_{a=1}^r  d_{G}(v_{j_a}) \gamma_a^2}\\
&\;\;\;\;\;\;\;\;\;\;\;\;\;\;\;\;\;\;\;\;\;\;\;\;\;\;\;\;\;\;\;\;\;\;\;\;\;\;\;\;\;\; \ge (1-\delta_r + \varepsilon) \sqrt{L_p} \left(1- \frac{\e^2}{(1-\delta_r)^2}  \right) 
\end{align*}
} 
Thus, putting things together, for $n$ large enough, 
$$\frac{\langle \phi, A(G) \phi \rangle}{\|\phi\|_2^2}\ge (1-\delta_r + \varepsilon) \sqrt{L_p} \left(1- \frac{\e^2}{(1-\delta_r)^2}  \right) -\frac{r \left(\e^2 \sqrt{L_p}  + 1 \right)}{(1-\delta_r)^2}\ge (1-\delta_r + \varepsilon/2) \sqrt{L_p},$$
since,  $\frac{ (1-\delta_r + \varepsilon)  \e^2}{(1-\delta_r)^2} \leq \frac{ 2  \e^2}{(1-\delta_r)} \leq \frac{\varepsilon}{4}$ and $\frac{r \e^2 }{(1-\delta_r)^2} \leq \frac{\varepsilon}{8}$, by using $\e \le (1-\delta_r)^2/8 r \leq (1-\delta_r)/8$ and $\frac{r}{\sqrt{L_p}(1-\delta_r)^2} \leq \frac{\varepsilon}{8}$, by choosing $n$ large enough (because $L_p \rightarrow \infty$ in the regime \eqref{eq:rangep}).  This completes the proof of the lemma.

\begin{remark}\label{remark:lt_marginal}
{\em Note that in the proof above (recall Lemma \ref{lm:r_I}) we have crucially used the fact that the lower tail event $\mathrm{LT}_r(\bm \delta)$ includes the marginal lower tail event for $\lambda_1(\cG_{n,p})$. Here, unlike in the case of the upper tail (recall Remark \ref{remark:ut_marginal}), analyzing the marginal lower tail behavior of the edge eigenvalues, other than the largest one, seems delicate. For concreteness let us only consider $\lambda_{2}(\cG_{n,p})$. While a marginal upper tail event of $\lambda_2(\cG_{n,p})$ also forces a similar event of $\lambda_1(\cG_{n,p})$, such an implication is absent in the lower tail setting. The key difficulty in this case is caused by the lack of monotonicity, that is, even though $\lambda_1(\cG_{n,p})$ is a monotone function of the edges of the graph, $\lambda_2(\cG_{n,p})$ is not (for instance the second eigenvalue of the complete graph is negative). This rules out the crucial application of FKG inequality that our current proofs rely on, and leaves open the problem of marginal large deviation for the second and smaller edge eigenvalues in the lower tail. } 
\end{remark}

\section{Refinements and Open Problems for the Largest Eigenvalue}
\label{sec:lambda1}

In the results above we derived the large deviations of the edge eigenvalues of $\cG_{n, p}$, in the regime $p \ll  \sqrt{\log n/\log \log n}/n$. It is now natural to wonder what happens when $p \gg \sqrt{\log n/\log \log n}/n$. Note that, as alluded to in \eqref{eq:lambda1_N}, the typical value of $\lambda_1(\cG_{n,p})$ and the corresponding eigenvector undergoes a transition at  $\sqrt{\log n/\log \log n}/n$: below the threshold the typical value of $\lambda_1(\cG_{n,p})$ is  asymptotically $\sqrt{L_p}$, which is governed by the maximum  degree of the graph, and corresponding eigenvector is localized; whereas above this threshold the typical value is asymptotically $np$, which is dictated by the total number of edges in the graph, and the corresponding eigenvector is completely delocalized. As a consequence, the large deviations behavior of $\lambda_1(\cG_{n,p})$ above the threshold is often very different.  In the following we summarize the existing results,  refine a result of Cook and Dembo \cite{CD18} regarding the lower tail of $\lambda_1(\cG_{n,p})$, and mention various open problems in this regime.

\subsection{Lower Tail for the Largest Eigenvalue}

Here, we discuss what happens to the lower tail probability $\lambda_1(\cG_{n, p})$, when $p \gg \sqrt{\log n/\log \log n}/n$. In this direction, the results of Lubetzky and Zhao \cite[Proposition 3.9]{LZ-dense} in the dense case, and, more generally, Cook and Dembo \cite[Theorem 1.21]{CD18}, give the precise asymptotics of the logarithm of the lower tail probability for $\lambda_1(\cG_{n, p})$, for $p \gg \log n/n$. In this regime, the problem exhibits replica symmetry, that is, the structure of the random graph conditioned on the lower tail event is approximately like another Erd\H os-R\'enyi random with a smaller edge density. }Therefore, given that the results in the previous sections cover the regime $p \ll \sqrt{\log n/\log \log n}/n$, the question that remains is, what happens to the lower tail of $\lambda_1(\cG_{n, p})$, in the intermediate regime $\sqrt{\log n/\log \log n}/n \ll p \lesssim \log n/ n$?  In the proposition below we show that this replica symmetric behavior, in fact, extends to all way down to $p \gg \sqrt{\log n/\log \log n}/n$.

\begin{ppn}\label{ppn:lt_lambda1} 
For $\sqrt{\log n/\log \log n}/n \ll p \leq \frac{1}{2}$ and $0 < q < p $ (such that $s:=q/p \in (0, 1)$ is fixed), 
\begin{align}\label{eq:lambda1_lt}
\P(\lambda_1(\cG_{n, p})  \leq q(n-1))=e^{-(1+o(1)) {n \choose 2} I_p(q)},
\end{align}
where $I_p(x) := x \log\frac{x}{p}+ (1-x) \log\frac{1-x}{1-p}$. 
\end{ppn}

The proof of this result is given in Appendix \ref{sec:largest_lt_pf}. The upper bound, which holds for all $0 < p < 1$, follows from  \cite[Theorem 1.18]{CD18}. The result in \cite[Theorem 1.18]{CD18} also gives the matching lower bound for $\log n/n \ll p \leq \frac{1}{2}$. In Appendix \ref{sec:largest_lt_pf}, we show that the proof of their lower bound, which proceeds by a measure titling argument, extends to $p \gg \sqrt{\log n/\log \log n}/n$ as well, because the typical value of $\lambda_1(\cG_{n, p})$ is asymptotically $np$, whenever $p \gg \sqrt{\log n/\log \log n}/n$. Note that this result together with Theorem \ref{thm:lt}, which covers the regime $p \ll \sqrt{\log n/\log \log n}/n$, which resolves the lower tail large deviation problem for the largest eigenvalue both above and below the threshold. 

\subsection{Upper Tail for the Largest Eigenvalue}

We end with a discussion on the upper tail of  $\lambda_1(\cG_{n, p})$, when $p \gg \sqrt{\log n/\log \log n}/n$. 
This problem in the regime $p \gg 1/\sqrt n$ was treated recently in \cite{uppertail_eigenvalue}, where arguments proceeded by observing that atypically large values of $\lambda_1(\cG_{n,p})$ forces atypically large cycle homomorphism counts.  The proof then relied on the recent progress in understanding mean field behavior of such observables, particularly the results in \cite{BGLZ,CD16}, which yielded 
\begin{align}\label{eq:ut_sqrtn}
-\log \P(\lambda_1(\cG_{n, p})\ge (1+\delta) )=(1+o(1)) \min \left\{\frac{(1+\delta)^2}{2}, \delta(1+\delta)\right\} n^{2}p^{2}\log (1/p).
\end{align}
Here, unlike in Theorem \ref{thm:1ut}, the atypical constructions corresponding to the two terms in \eqref{eq:ut_sqrtn} above are obtained by planting a clique (of size $O(np)$) or an anti-clique (of size $O(np^2)$) in the random graph.\footnote{Given a graph $G=(V(G), E(G))$, a set $S \subseteq V(G)$ is said to be an anti-clique in $G$ if every vertex in $S$ is connected to every other vertex in $G$.} Using the above result and observing that the anti-clique construction ceases to be viable for $p\ll 1/\sqrt n,$ one can speculate that even for $\sqrt{\log n/\log \log n}/n \ll p \ll 1/\sqrt n,$ 
\begin{equation}\label{predic1}
-\log \P(\lambda_1(\cG_{n, p})\ge (1+\delta)np )=(1+o(1)) \frac{(1+\delta)^2}{2} n^{2}p^{2}\log (1/p).
\end{equation}
 On the other hand, Theorem \ref{thm:1ut} suggests an alternate strategy of planting a vertex of degree $(1+\delta)^2 n^2p^2,$ to get the required increase in the spectral norm.
It is plausible that the above two predictions govern $-\log \P(\lambda_1(\cG_{n, p})\ge (1+\delta)np)$ in the entire regime $\sqrt{\log n/\log \log n}/n \ll p \ll 1/\sqrt n$, and addressing this formally will be taken up in a future project. 

\small 
\subsection*{Acknowledgements} The authors thank Anirban Basak and Amir Dembo for several useful discussions. SG is partially supported by NSF grant DMS-1855688, NSF CAREER Award DMS-1945172 and a Sloan Research Fellowship.

\normalsize

\bibliographystyle{plain}
\bibliography{ldp_ref}

\appendix

\section{Proofs of Technical Lemmas}

\subsection{Proof of Lemma \ref{equivalence}} 
\label{sec:sqrtLp_pf} 

Recall from \eqref{eq:Lp} that $L_p=\frac{\log n}{\log \log n -\log (np)}$.  Let $I_1$ be the set of $p$ such that $n p \ll \sqrt{\log n/\log \log n}$ and $I_2$ be the set of $p$ such that $np \ll \sqrt{L_p}$. 

To begin with, we show $I_1 \subset I_2$. To this end, first assume $np > 1/\log n.$ Then, if $p \in I_1,$ then noting, $-\log\log n<\log (np)<\frac{1}{2} \log\log n$, it follows that $\sqrt{L_p}=\Theta(\sqrt{\log n/\log \log n})$, and, hence, $p\in I_2.$ Now, suppose $np\le 1/\log n$. In this case, $L_p = \Theta((\log n/\log(1/np)).$ Moreover, notice that since $np\le 1/\log n,$ for all large enough $n$, $$np\sqrt{\log(1/np)}\le 1\ll \sqrt{\log n} \implies np\ll \sqrt{L_p},$$ 
which shows $p \in I_2$. 

Next, we show $I_2 \subset I_1$, when $np\ll \log n$. To this end, suppose $p \in I_2$. Then, by the fact $np\ll \log n$, $\log\log(n)-\log(np)=\omega(1)$ and, hence, $L_p=o(\log n)$. Thus,  we have the following implications:
\begin{align*}
p\in I_2 &\implies np \ll \sqrt{\log n}
\implies \log(np)\le \frac{1}{2} \log\log n 
\implies L_p \lesssim \log n/\log \log n
\implies p\in I_1.
\end{align*}

Finally, we show that $np \ll \log n$ implies $np \ll L_p$. To this end, assume $np =c\log n$, where $c=o(1)$. Then to show $np \ll L_p$, it is enough to prove $c\log n\ll \frac{\log n}{\log (1/c)}$. This follows by observing that $c \log (1/c)\ll 1$, since $c=o(1).$

\subsection{Proof of Lemma \ref{lem:degreep}}
\label{sec:degreep_pf}

To begin with suppose $\log n \leq C \log (1/np)$, for some constant $C > 0$. Then, for $s \geq C+1$,  
$$n \binom{n-1}{s}p^s(1-p)^{n-s} \leq n n^s p^s = e^{\log n - s \log(1/np)} \leq e^{-\frac{1}{C} \log n} \ll 1.$$
Then recalling the definition in \eqref{eq:degreep} it follows that $\Delta_p \leq C $.  This completes the proof of part (a). 

For part (b) choose $\varepsilon_n=1/\sqrt{\log (\log n /np)}$. 
Then for any $s \geq (1+\varepsilon_n) L_p$, $$ s\log s \geq  (1+\varepsilon_n) L_p\log((1+\varepsilon_n) L_p) = L_p\log L_p +\varepsilon_n L_p\log L_p + O(L_p),$$ and $s \log (1/np) \geq  L_p\log(1/np) +\varepsilon_n L_p\log(1/np)$. This gives, using the fact $L_p \gg 1$ (which follows from the assumption $\log n \gg \log(1/np)$) and $p \ll \frac{\log n }{n}$, 
\begin{align}
n \binom{n-1}{s}p^s(1-p)^{n-s} \leq n \frac{n^s e^s}{s^s} p^s & = e^{\log n - s \log(1/np) - s \log s  +s } \nonumber \\ 
& \leq e^{\log n - L_p\log (L_p/np) - \varepsilon_n L_p\log (L_p/np)   + O(L_p) } \nonumber \\ 
\label{eq:deltap_I} & = e^{ L_p\log (\log( \log n/np)) - \varepsilon_n L_p\log (L_p/np)  + O(L_p) } \\ 
& \leq e^{ 2 L_p\log (\log (\log n/np) ) - \varepsilon_n L_p\log (\log n /np)  + O(L_p) }  \nonumber \\ 
& \leq e^{ - \Theta\left( L_p \sqrt{\log (\log n /np)}  \right)  } \ll 1, \nonumber 
\end{align}
where in the first step we use the bound $\binom{n}{s} \le \frac{n^s e^s}{s^s}$, and  \eqref{eq:deltap_I} uses $$L_p\log (L_p)-\frac{\log n \log \log n}{\log \log n+\log(1/np)} = -\frac{\log n \log (\log \log n+ \log(1/np))}{\log \log n+\log(1/np)}=-L_p \log (\log (\log n/np)),$$ and the observation that  $\log n - \frac{\log n \log \log n}{\log \log n + \log (1/np)} -L_p \log(1/np) =0 $. This shows, 
\begin{align}\label{eq:deltaub}
\Delta_p \leq  (1+\varepsilon_n) L_p. 
\end{align} 
For the lower bound, suppose $s \leq (1-\varepsilon_n) L_p$ and choose $n$ large enough such that $(1-p)^{n-s} \geq (1-p)^n \geq e^{-O(np)} \ge e^{-O(L_p)}$, where the last inequality uses Lemma \ref{equivalence}. Thus, similarly as above,  
\begin{align*}
n \binom{n-1}{s}p^s(1-p)^{n-s} \geq  \tfrac{1}{2} n \frac{n^s}{s^s} p^s & = e^{\log n - s \log(1/np) - s \log s   -O(L_p)} \nonumber \\ 
& \geq \tfrac{1}{2} e^{\log n - L_p\log (L_p/np) + \varepsilon_n L_p\log (L_p/np)    -O(L_p) }\nonumber \\ 
& =  e^{ L_p\log (\log( \log n/np)) + \varepsilon_n L_p\log (L_p/np)  - O(L_p) } \\ 
& \geq e^{ \tfrac{1}{2} L_p\log (\log (\log n/np) ) + \varepsilon_n L_p\log (\log n /np)  + O(L_p) }  \nonumber \\ 
& \geq e^{  \Theta\left( L_p \sqrt{\log (\log n /np)}  \right)  } \gg 1, \nonumber 
\end{align*}
This shows $\Delta_p \geq  (1 - \varepsilon_n) L_p$.  This, combined with \eqref{eq:deltaub}, completes the proof of part (b).

Note that for $1 \lesssim n p \ll  \log n$, the result in (c) follows from \cite[Proposition 1.13]{bbk}. Therefore, it suffices to assume that $ n p \ll 1$ and $\log n \gg \log (1/np)$.  The upper bound follows from the typical value $d_{(1)}(\cG_{n, p})$ \cite[Lemma 2.2]{KS} and  part (b) above, since, for $1 \leq a \leq r$,  
\begin{align}\label{eq:lambda_bound}
d_{(a)}(\cG_{n, p}) \leq d_{(1)}(\cG_{n, p}) = (1+o(1)) \Delta_p = (1+o(1)) L_p.
\end{align}
The lower bound follows from the proof of Proposition \ref{ppn:deglt}, which shows, for any $\varepsilon > 0$, $\P(d_{(a)}(\cG_{n, p}) \leq (1-\varepsilon) L_p) \rightarrow 0$, for $1 \leq a \leq r$.

As before, for $1 \lesssim n p \ll  \sqrt{L_p}$, the result in (d) follows from \cite[Proposition 1.9]{bbk}. Therefore by Lemma \ref{equivalence}, it suffices to assume that $ n p \ll 1$ and $\log n \gg \log (1/np)$.  The upper bound follows from typical value $ \lambda_1(\cG_{n, p})$ \cite[Theorem 1.1]{KS} and  part (b) above, 
\begin{align}\label{eq:lambda_bound}
\lambda_a(\cG_{n, p}) \leq \lambda_1(\cG_{n, p}) \leq (1+o(1)) \sqrt{\Delta_p} = (1+o(1)) \sqrt{L_p}.
\end{align}
For the lower bound, suppose  $G \sim  \cG_{n, p}$ be a realization of the random graph, and $\cT$ be the event that $G$ is a disjoint union of trees and each component of $G$ is of size at most $(1+o(1)) \Delta_p$. Also, let $\cD_r$ be the event that $d_{(a)}(G)=(1+o(1)) L_p$, for $1 \leq a \leq r$. Recall from \cite[Lemma 2.2]{KS} and part (c) above that $\P(\cT \cap \cD_r) \rightarrow 1$. Now, assume $G \in \cT \cap \cD_r$, and suppose $G_1, G_2, \ldots, G_\nu$ be the connected components of $G$. To begin with, suppose there is a $1 \leq b \leq \nu$ such that 2 vertices in $G_b$ have degree greater than  $(1+o(1)) \Delta_p$. Then the induced subgraph of $G$ on the set of vertices $\{u, v\} \bigcup N_{G_b}(u) \bigcup N_{G_b}(v)$ will have size greater than $(1+o(1)) 2 \Delta_p$, which is a contradiction. Therefore, each connected component of $G$ has at most one vertex of degree $(1+o(1)) \Delta_p= (1+o(1)) L_p$, which means, by part (c), the vertices 
$i_1, i_2, \ldots, i_r$ which attain the $r$ largest degrees of $G$ will belong to  different connected components of $G$. Assume, without loss of generality, these components are $G_1, G_2, \ldots, G_r$. This implies, by Lemma \ref{lem:gfact}(a), 
$$\lambda_1(G_a) \geq \sqrt{d_{(1)}(G_a)} = (1+o(1))\sqrt{L_p}, \quad \text{for} \quad 1 \leq a \leq r.$$ 
Since, $G_1, G_2, \ldots, G_r$ are disjoint, the multi-set $\{\lambda_1(G_1), \lambda_1(G_2), \ldots, \lambda_1(G_r)\}$ is a subset of the spectrum of $G$. This combined with the upper bound \eqref{eq:lambda_bound}, implies, for $1 \leq a \leq r$, $\lambda_a(\cG_{n, p})=(1+o(1)) \sqrt{L_p}$, with high probability. \hfill $\Box$ \\

\subsection{Proof of Lemma \ref{lem:largecycle}}
\label{sec:largecycle_pf}
Recall, from \eqref{eq:Lp}, $L_p = (1+o(1))(\log n /(\log\log n-\log(np)))$. To begin with consider the case $e^{-(\log \log n)^2} \leq np \leq \log^{1/4}n$. Assume there exists a set $T \subset V(\cG_{n, p})$ such that $|T|=t$ and $\min_{v \in T} d_v(\cG_{n, p}) \geq \log^{1/3}n / \log \log n$. If such a set $T$ exists, then either $|E(T, V(\cG_{n, p}) \setminus T)|$ or the $|E(\cG_{n, p}[T])|$ is at least $e_0 := (\log^{1/3}n / \log \log n)t/3$.\footnote{Given a graph $G=(V(G), E(G))$ and $S \subseteq V(G)$, $E(S, V(G) \backslash S)$  denotes the collection of edges across the cut $(S, V(G) \backslash S)$. Moreover, $G[S]$ denotes the induced subgraph $G$ with vertex set $S$.} Since 
$$|E(T,V(\cG_{n, p}) \setminus T)| \sim \dBin(t(n-t),p)  \quad \text{and} \quad |E(\cG_{n, p}[T])|\sim \dBin(t(t-1)/2,p),$$ the probability that they are at least $e_0$ is at most  $\binom{t(n-t)}{e_0}p^{e_0}$ and $\binom{t(t-1)/2}{e_0}p^{e_0}$, respectively. Plugging in the value of $e_0$, observe that both these quantities are upper bounded by	
\begin{equation*}
\left(\frac{3enp \log \log n}{\log^{1/3} n}\right)^{e_0} \leq \left(\frac{3e \log^{1/4}n \log \log n}{\log^{1/3} n}\right)^{\frac{\log^{1/3}n}{\log\log n}t}\leq e^{{-\Omega(t\log^{1/3}n)}}.
\end{equation*}
This shows, since $\log^{1/3}n / \log \log n \ll \log^{1/3}n /( \log \log n )^{2/3}   \lesssim  \Delta_p^{1/3} + np(1+1/\log\log n)$ by Lemma \cite[Lemma 2.1]{KS}, the probability that all vertices of a set of size $t$ has degree at least $\Delta_p^{1/3} + np(1+1/\log\log n)$ is at most $e^{{-\Omega(t\log^{1/3}n)}}$. Further, similar argument shows that conditioning on the presence of a fixed set of at most $2t$ edges in $\cG_{n, p}$, the desired probability is still at most $e^{{-\Omega(t\log^{1/3}n)}}$. Hence, the probability that there exists a cycle of length $s \geq L \geq \log^{5/6} n$ with at least $s/2$ vertices inside the set $X$ is at most
\begin{align}\label{eq:a1}
\sum_{s \geq L} n^s p^s \binom{s}{\lceil \frac{s}{2} \rceil} e^{-\Omega(\frac{s}{2}\log^{1/3}n)} \leq \sum_{s \geq L} (2np e^{-\Omega(\log^{1/3}n))})^s & \leq e^{-\Omega(L\log^{1/3}n)} \nonumber \\
& = e^{-\Omega(\log^{7/6} n)},
\end{align}
where LHS is obtained by fixing the $s$-vertices and $s$-edges of the cycle, and selecting a set $T$ of size $t= \lceil s/2 \rceil$ and considering the event that conditioned on the presence of the cycle, all vertices of $T$ belong to the set $X$. 

Now, consider $np \geq \log^{1/4}n$. Here, we estimate the probability that there exists a set of vertices $T$ with $|V|=t \leq n/2$ with $\min_{v \in V(T)} d_v(G) \geq np(1+1/\log\log n)$. Considering $|E(T, V(\cG_{n, p}) \setminus T)|$ and $|E(\cG_{n, p}[T])|$ as before, we obtain that such a set $T$  exists with probability at most $e^{-\Omega(tnp/(\log\log n)^2)}$. Also, conditioning on presence of $2t$ specific edges does not change the order of exponent. Hence, the probability that there is a cycle of length $s \geq L$ with $s/2$ edges in $X$ is at most
\begin{align}\label{eq:a2}
\sum_{s \geq L} n^s p^s \binom{s}{\lceil \frac{s}{2} \rceil} e^{-\Omega(\frac{s}{2} np/(\log\log n)^2)} &\leq \sum_{s \geq L} (2np e^{-\Omega(np/(\log\log n)^2))})^s \nonumber \\ &\leq e^{-\Omega(L\log^{1/4}n/(\log\log n)^2)}=  e^{-\Omega(\log^{17/16} n)}. 
\end{align}

Finally, following the proof of \cite[Lemma 2.3]{KS}, it is easy to see that $\P(\cA_2) = e^{-\Omega(\log^{\frac{17}{16}} n)}$. This combined \eqref{eq:a1} and \eqref{eq:a2}, shows that $\P(\cA_1\bigcup \cA_2) = e^{- \omega(\log n)}$, as required.

\section{Proof of Propostion \ref{ppn:lt_lambda1}} 
\label{sec:largest_lt_pf}

The upper bound holds for all $0 < q < p < 1/2$ by \cite[Theorem 1.18]{CD18}, 
$$\P(\lambda_1(\cG_{n, p})  \leq q(n-1)) \leq e^{- {n \choose 2} I_p(q)}.$$
The matching lower bound was also proved by \cite[Theorem 1.18]{CD18} for $\log n/n \ll p \leq \frac{1}{2}$. Here, we show that the same argument extends to $\sqrt{\log n/\log \log n} \leq n p \leq 
\log n$ with necessary modifications. The proof follows by a standard tilting argument. Fix an $\e>0$ small enough and let $q'=(1-\e)q.$ Denote by $\Q$ the measure on $\sG_n$ induced by $G(n, q')$ and $\mathrm{LT}(q)$ the event $\{G\in \sG_n: \lambda_1(G)  \leq q(n-1)\}$. We begin by defining the proxy event $\cE$ which implies $\mathrm{LT}(q)$: 
\begin{align} \label{proxy}
\cE & := 
\left\{ 
\begin{array}{cc}
\left\{\|A(\cG_{n,q'})-q'{J}_n\|\le K \sqrt{\frac{\log n}{\log\log n}}\right\}   &  \text{ for }   \sqrt{\log n/\log \log n} \ll nq' <  \log^{0.9}n,  \\
\left\{\|A(\cG_{n,q'})-q'{J}_n\|\le K \sqrt{\log^{1.1} n} \right\}   &  \text{ for }   \log^{0.9}n\le  nq'\le \log n, 
\end{array}
\right. 
\end{align}
where $J_n$ is the $n \times n$ matrix with zeros on the diagonal and 1 everywhere else. 

\begin{lem} $\Q(\cE)\ge \frac{1}{2}.$ 
\end{lem}

\begin{proof} To begin with suppose  $\sqrt{\log n/\log \log n} \ll nq' <  \log^{0.9}n$. In this case, by Lemma \ref{lem:degreep}(c), $d_{(1)}(\cG_{n, q'})=(1+o(1)) L_{q'}$ and $L_{q'} = \Theta( \log n/\log \log n)$. Then, by Theorem \ref{thmversh}, applied with $\cG_{n, p}^{-}=\cG_{n, p}$, $\Q(\cE)\ge \frac{1}{2}$. 
Next, suppose $\log^{0.9}n\le  nq'\le \log n$. Here, by a simple coupling argument and using  $d_{(1)}(\cG_{n, r})=(1+o(1))nr$, for $nr \geq \log^{1.1} n$ (which can be proved by using standard binomial concentration  (cf. \cite{KS})), it follows that $d_{(1)}(\cG_{n, q'}) \leq (1+o(1)) \log^{1.1} n$. Theorem \ref{thmversh} then implies $\Q(\cE)\ge \frac{1}{2}$. (Note that the reason behind our definition in \eqref{proxy}, in the case $\log^{0.9} n\le nq'\le \log n$ is that the behavior of $d_{(1)}(\cG_{n, q'})$ as $nq'=c\log n$ for different values of $c$ is a bit delicate exhibiting a transition from small to large values of $c.$ Though this can be pinned down using arguments similar to \eqref{eq:deltap_I}, or using the more precise results in \cite{edge_knowles}, the crude bound in \eqref{proxy} suffices for our purpose.)
\end{proof}

Now, observe that the hypothesis $n q \gg \sqrt{\log n/\log \log n} $ and triangle inequality for the spectral norm imply $\cE\subset \mathrm{LT}(q).$ Thus, 
\begin{align}\label{tilt}
\P(\mathrm{LT}(q))\ge  \P(\cE) = \int_{\cE} e^{-\log \frac{\mathrm d \Q}{\mathrm d \P}} \mathrm d \Q = \Q(\cE) \left(\frac{1}{\Q(\cE)} \int_{\cE} e^{-\log \frac{\mathrm d \Q}{\mathrm d \P}} \mathrm d \Q \right). 
\end{align}
Next, note that for $G \in \sG_n$ with adjacency matrix $A(G)=((a_{ij}))_{1 \leq i, j \leq n}$, 
$$\log \frac{\mathrm d \Q}{\mathrm d \P}(G)- \binom{n}{2}I_{p}(q')=\kappa(q',p)\sum_{1\le i < j \le n}(q'-a_{ij}),$$
where $\kappa(q',p)=\log \frac{1-q'}{1-p}+\log \frac{p}{q'}\le \log \left(\frac{2p}{q'}\right),$ which under our hypothesis is a constant.
Then as in \cite{CD18}, we use the bound $$\sum_{1\le i < j \le n}(q'-a_{ij})=\frac{1}{2} {\bm 1}^{\top}(q'{J}_n-A(\cG_{n,q'})){\bm 1} \le \frac{n}{2}\|q'{J}_n-A(\cG_{n,q'})\|\le n\alpha (q'),$$
where $\alpha(q')=\sqrt{ \log n / \log\log n }$ when $\sqrt{\log n/\log \log n} \ll nq' \le \log^{0.9}$, and $\alpha(q')=\sqrt{\log^{1.1} n}$ when $\log^{0.9} \leq n q' \leq \log n$. Thus, on $\cE,$ we get
$\log \frac{\mathrm d \Q}{\mathrm d \P}(G)- \binom{n}{2}I_{p}(q')\le n\alpha(q').$ Plugging this into \eqref{tilt} yields $\P(\mathrm{LT}(q))\ge \frac{1}{2}e^{-{n \choose 2} I_p(q')-n\alpha(q')}.$
A simple calculation yields $I_{p}(q')= (1+O(\e))I_{p}(q),$ where the last equality holds uniformly over all small $\e,$ and $p,q\le 1/2$ where $q/p$ is a fixed constant. Finally, the observation that $I_{p}(q) = \Theta(q) = \Theta(p)$ and $np \gg \alpha(p)$, for all $\sqrt{\log n/\log \log n}\ll np \le \log n$,  which follows by definition, implies 
$$e^{-{n \choose 2} I_p(q')-n\alpha(q')}\ge e^{-(1+O(\e)){n \choose 2} I_p(q)},$$ which completes the proof since $\e$ was arbitrary. 
\end{document}